\documentclass[11pt]{article}
\usepackage{amsthm,amsmath,amsfonts,amssymb}
\usepackage{color}

\usepackage{geometry}
\usepackage{enumitem}

\newtheorem{theorem}{Theorem}

\newtheorem{corollary}[theorem]{Corollary}

\newtheorem{definition}[theorem]{Definition}

\newtheorem{lemma}[theorem]{Lemma}

\newtheorem{proposition}[theorem]{Proposition}
\newtheorem{remark}[theorem]{Remark}

\numberwithin{equation}{section}
\numberwithin{theorem}{section}

\newcommand{\Id}{{\mathrm{I}}}
\newcommand{\nv}{\boldsymbol}

\newcommand{\1}[1]{{\boldsymbol 1_{\{#1\}}}}
\newcommand{\oo}{{\boldsymbol 1}}

\renewcommand{\P}{{\mathbb P}}

\newcommand{\R}{{\bf R}}

\renewcommand{\S}{{\bf S}}
\newcommand{\M}{{\bf M}}

\newcommand{\Ccal}{{\mathcal C}}

\newcommand{\Ecal}{{\mathcal E}}

\newcommand{\Lcal}{{\mathcal L}}

\newcommand{\CNeu}{{\mathcal D}}
\newcommand{\HNeu}{{H_{\rm Neu}^{1,2}}}

\DeclareMathOperator{\Diag}{Diag}

\DeclareMathOperator{\tr}{Tr}
\DeclareMathOperator{\rk}{rank}
\DeclareMathOperator{\adj}{adj}

\begin{document}

\title{Matrix-valued Bessel processes\thanks{The author would like to thank Dmitriy Drusvyatskiy, Damir Filipovi\'c, Piotr Graczyk, Eberhard Mayerhofer, Jim Renegar and Lioudmila Votrikova for stimulating discussions and several comments that led to substantial improvements of this manuscript. Thanks are also due to an anonymous referee whose detailed comments led to several improvements. This research was undertaken while the author was at the Swiss Finance Institute at EPFL and was funded in part by the European Research Council under the European Union's Seventh Framework Programme (FP/2007-2013) / ERC Grant Agreement n. 307465-POLYTE.}}
\author{Martin Larsson\thanks{ETH Zurich, Departement of Mathematik, Switzerland. Email: martin.larsson@math.ethz.ch.}}
\date{}
\maketitle

\begin{abstract}
This paper introduces a matrix analog of the Bessel processes, taking values in the closed set $E$ of real square matrices with nonnegative determinant. They are related to the well-known Wishart processes in a simple way: the latter are obtained from the former via the map $x\mapsto x^\top x$. The main focus is on existence and uniqueness via the theory of Dirichlet forms. This leads us to develop new results of potential theoretic nature concerning the space of real square matrices. Specifically, the function $w(x)=|\det x|^\alpha$ is a weight function in the Muckenhoupt~$A_p$ class for $-1<\alpha\le 0$ ($p=1$) and $-1<\alpha<p-1$ ($p>1$). The set of matrices of co-rank at least two has zero capacity with respect to the measure $m(dx)=|\det x|^\alpha dx$ if $\alpha>-1$, and if $\alpha\ge 1$ this even holds for the set of all singular matrices. As a consequence we obtain density results for Sobolev spaces over (the interior of) $E$ with Neumann boundary conditions. The highly non-convex, non-Lipschitz structure of the state space is dealt with using a combination of geometric and algebraic methods.
\end{abstract}

\section{Introduction and preliminaries}

The Wishart processes, taking values in the cone $\S^d_+$ of positive semidefinite $d\times d$~matrices, constitute a class of matrix-valued Markov processes generalizing the squared Bessel (BESQ) processes. They were first introduced by Bru~\cite{Bru:1989,Bru:1991}, and have subsequently been studied further and extended in various directions by a number of authors, for example~\cite{Donati-Martin:2004vn,Graczyk/Vostrikova:2007,Donati-Martin:2008,Cuchiero:2011uq}. They have also found use in applied contexts, for instance in finance~\cite{Gourieroux/Sufana:2007,DaFonseca/etal:2008}.

The existence of a well-behaved matrix analog of the BESQ processes raises the question of whether the same is true for the Bessel (BES) processes. Since the Wishart process is $\S^d_+$-valued, a natural candidate is its positive semidefinite square root. This was considered in~\cite{Graczyk/Malecki:2012}, where the resulting Markov process is described via the dynamics of its eigenvectors and eigenvalues. However, as was pointed out already by Bru~\cite{Bru:1991}, it appears difficult to obtain the dynamics of the process itself, or to succinctly describe its generator.

The aim of the present paper is to show that a more well-behaved class of processes is obtained by passing to the larger state space
\[
E = \{x \in \M^d: \det x \ge 0\},
\]
where $\M^d$ is the Euclidean space of all $d\times d$ real matrices, endowed with the usual inner product $x\bullet y=\tr(x^\top y)$ and norm $\|x\|=\sqrt{x\bullet x}$. The \emph{matrix-valued Bessel process with parameter $\delta>0$ and matrix dimension $d$}, abbreviated ${\rm BESM}(\delta,d)$, will be an $E$-valued Markov process whose generator is given by
\begin{equation} \label{eq:A}
\Lcal f = \frac{1}{2} \Delta f + \frac{\delta-1}{2}  x^{-\top} \bullet \nabla f,
\end{equation}
where $x^{-\top}=(x^{-1})^\top$, $\nabla$ is the $d\times d$ matrix with elements $\partial_{x_{ij}}$ (so that for $f\in C^1(\M^d)$, $\nabla f$ is its gradient), and $\Delta = \nabla\bullet\nabla=\sum_{i,j}\partial^2_{x_{ij}x_{ij}}$ is the Laplacian. To make $x^{-1}$ globally defined, we set to zero for singular $x$ (this choice is arbitrary and inconsequential.) Notice that for $d=1$, $\Lcal$ is the generator of the ${\rm BES}(\delta)$ process.

Existence of the BESM$(\delta,d)$ process is proved via the theory of Dirichlet forms, which is able to nicely handle the singular drift term of $\Lcal$. The crucial fact is that $\Lcal$ is a symmetric operator with respect the measure $m(dx)=|\det x|^{\delta-1}dx$, which is a consequence of an integration by parts formula (Theorem~\ref{T:ibp}). The Dirichlet form is then given by the simple expression
\[
\Ecal(f,g) = \int_E \nabla f \bullet \nabla g\ m(dx).
\]

Uniqueness is a much more delicate issue. Relying on density results for certain Sobolev spaces with Neumann boundary condition (Theorem~\ref{T:MU}), we establish {\em Markov uniqueness} in the sense of Eberle~\cite{Eberle:1999uq}. Obtaining these density results is a nontrivial matter. In particular, we are led to prove several results, interesting in their own right, about the measure $m$ and its interaction with the state space. Specifically, we show that the matrices of co-rank at least two form a set of zero capacity with respect to~$m$, and that if $\delta\ge 2$, the set of all singular matrices has zero capacity (Theorem~\ref{T:cap0}). Moreover, we prove that $|\det x|^\alpha$ is locally Lebesgue integrable on $\M^d$ precisely when $\alpha>-1$ (Theorem~\ref{T:Radon}), and that it is a weight function in the Muckenhoupt~$A_p$ class when $\alpha\in(-1,0]$ and $p=1$, and when $\alpha\in(-1,p-1)$ and $p>1$ (Theorem~\ref{T:Ap}). This exactly parallels the well-known situation for the weight function~$t^\alpha$ on~$\R$.

The proofs of these results require some effort. The difficulties mainly arise due to the highly non-convex, non-Lipschitz structure of the state space~$E$. In fact, the interior~$E^o$ does not even lie on one side of its boundary $\partial E$, as can be seen by considering the lines through the origin, $\{tx:t\in\R\}$, $x\in\M^d\setminus\{0\}$: If $d$ is even, each line lies either entirely inside $E$, or entirely outside $E^o$. These issues are resolved via a combination of geometric methods (relying on the stratification of $E$ into smooth manifolds $M_k$ consisting of rank~$k$ matrices) and algebraic methods (mainly the $QR$-decomposition and estimates of the determinant function near $M_k$.) One would expect similar techniques to be useful for the analysis of Markov processes on more general stratified spaces; some of the groundwork for this is laid in~\cite{Drusvyatskiy/Larsson:2015}, and the case of so-called polynomial preserving diffusions is treated in~\cite{FilipovicLarsson15}.

Let us say something about why the BESM processes are natural analogs of the BES processes, other than the resemblance of their generators. The main reason is that the process $X^\top X$, where $X$ is ${\rm BESM}(\delta,d)$, is a Wishart process with parameter $\alpha=d-1+\delta$, denoted ${\rm WIS}(\alpha,d)$, see Theorem~\ref{T:WIS}. As a consequence, $\|X\|$ is ${\rm BES}(d\alpha)$, and $\det X$ is a time-changed ${\rm BES}(\delta)$ process. Moreover, just as in the scalar case, $X$ is the weak solution of a stochastic differential equation for~$\delta>1$, while it is not even a semimartingale for $0<\delta<1$. A general discussion of BES and BESQ processes is available in~\cite[Chapter~XI]{Revuz:1999zr}. For a specialized treatment of the case $0<\delta<1$, see~\cite{Biane/Yor:1987,Bertoin:1990}.

A second motivation, which was the original ``clue'' that led us to consider the generator $\Lcal$, is as follows. Let $X$ be an $\M^d$-valued Brownian motion starting from $\Id$ (the identity matrix), and let $\tau_0$ be the first time $\det X_t$ hits zero. Then $\det X_{t\wedge\tau_0}$ is a martingale, and we may use it to change the probability measure. An application of Girsanov's theorem, using the identity $\nabla \ln \det(x) = x^{-\top}$, shows that under the new measure, the process
\[
W_t = X_t - X_0 - \int_0^t X^{-\top}_s ds, \qquad t\ge 0,
\]
is $\M^d$-valued Brownian motion. Thus $X$ becomes a Markov process whose generator is $\Lcal$ with $\delta=3$, and its determinant is positive by construction. This is fully analogous to the well-known construction via Doob's $h$-transform of the ${\rm BES}(3)$ process as Brownian motion ``conditioned to stay positive''.

In addition to the notation already introduced above, the following conventions will be in force throughout the paper.

\begin{enumerate}
\item[$\bullet$] As usual, the symbols $C(U)$; $C_c(U)$; $C^k(U)$ denote the spaces of continuous; continuous and compactly supported; $k$-times continuously differentiable functions on a subset $U\subset\M^d$ equipped with the relative topology. Writing $C(U;V)$, etc., means that the functions take values in the topological space $V$. Note that $U$ may be closed in $\M^d$, e.g.~if $U=E$. In this case, the compact sets need not be bounded away from~$\partial U$.

\item[$\bullet$] We set $M_k=\{x\in\M^d:\rk x=k\}$, $k=0,\ldots,d$. Then $M_k$ is a smooth manifold of dimension $d^2-(d-k)^2$, see~\cite{Helmke/Shayman:1992}, and we have $\partial E = \cup_{k\le d-1}M_k$. For any $v\in \M^d$ we say that \emph{$v$ is tangent to $\partial E$ at $x\in M_k\subset\partial E$} if $v$ lies in the tangent space of $M_k$ at $x$. We refer to~\cite{Lee:2003} for background on differential geometry.

\item[$\bullet$] $O(d)$ is the orthogonal group over $\R^d$, and $T(d)$ is the group of upper-triangular $d\times d$ real matrices with strictly positive diagonal entries. The set of nonsingular matrices (i.e., the general linear group) is homeomorphic to $O(d)\times T(d)$ via the $QR$-decomposition. The following change of variable formula, which is a consequence of the uniqueness of Haar measure, follows directly from \cite[Proposition~5.3.2]{Faraut:2008fk} and monotone convergence.

\begin{lemma} \label{L:c-o-v}
Let $f:\M^d\to \R$ be nonnegative and measurable. Then
\[
\int_{\M^d} f(x) |\det x|^{-d} dx = \int_{O(d)\times T(d)} f(QR) \mu(dQ) \prod_{i=1}^d R_{ii}^{-i} dR,
\]
where $\mu$ is proportional to normalized Haar measure on $O(d)$, and $dR=\prod_{i\le j}dR_{ij}$.
\end{lemma}

\item[$\bullet$] For $x\in\M^d$, we let $\adj x$ denote the \emph{adjugate matrix} of $x$ (i.e., the transpose of the matrix of cofactors). It satisfies the identities $x\adj x = (\det x)\Id$ and $\nabla\det(x)=\adj x$, so that in particular $x^{-\top}=\nabla\det(x)/\det x$ for nonsingular $x$. We also have $\adj x=0$ if and only if $\rk x\le d-2$.
\end{enumerate}

The rest of this paper is organized as follows. The BESM process and semigroup are defined, and proved to exist, in Section~\ref{S:DE}. Some fundamental properties, including the relation to the Wishart process, are discussed in Section~\ref{S:WIS}. Markov uniqueness is proved in Section~\ref{S:MU}. The crucial Theorems~\ref{T:cap0} and~\ref{T:Ap} are proved in Sections~\ref{S:cap0} and~\ref{S:A1}, respectively. The integration by parts formula (Theorem~\ref{T:ibp}) is proved in Appendix~\ref{A:ibp}, while Appendix~\ref{A:W12} and~\ref{A:tube} contain, respectively, some auxiliary results on Sobolev spaces and differential geometry.

\section{Definition and existence} \label{S:DE}

The definition of the BESM process is based on the differential operator $\Lcal$ in~\eqref{eq:A} acting on functions in $\CNeu$, where
\[
\CNeu = \left\{ f\in C^2_c(E) \ :\  x^{-\top}\bullet\nabla f \text{ is bounded} \right\}.
\]
As we will see momentarily, $(\Lcal,\CNeu)$ is a symmetric operator on $L^2(E,m)$, where the measure $m$ is given by
\[
m(dx) = |\det x|^{\delta-1}dx.
\]
(Occasionally $m$ will be viewed as a measure on $\M^d$, or on subsets other than $E$.) The inner product on $L^2(E,m)$ is denoted by $\langle \cdot,\cdot\rangle$. Specifically, we write
\[
\langle f,g\rangle = \int_E f(x)g(x)m(dx) \quad \text{and} \quad \langle F,G\rangle = \int_E F(x)\bullet G(x) m(dx),
\]
where $f,g\in L^2(E,m)$ and $F,G\in L^2(E,m;\M^d)$. The overlapping notation should not cause any confusion.

The following result shows that $m$ is a Radon measure on $\M^d$, and hence on $E$, when $\delta>0$. It implies in particular that we have $\CNeu\subset L^2(E,m)$.

\begin{theorem} \label{T:Radon}
Let $\alpha\in\R$ and define $w(x)=|\det x|^\alpha$. The function $w$ is locally integrable on $\M^d$ if and only if $\alpha>-1$.
\end{theorem}

\begin{proof}
Let $A \subset \M^d$ be relatively compact. Since $\partial E$ is a nullset, we may assume that $A\cap\partial E=\emptyset$. Then there is a rectangle $K\subset T(d)$, say $K=\times_{i\le j}I_{ij}$, of bounded open intervals $I_{ij}$ such that $I_{ii}\subset (0,\infty)$ and $I_{ij}\subset\R$ ($i<j$), and such that $A\subset O(d)\cdot K$. Hence, the change-of-variable formula in Lemma~\ref{L:c-o-v} yields
\begin{align*}
\int_A |\det x|^\alpha dx &\le \mu(O(d)) \bigg( \prod_{i<j} \int_{I_{ij}} dR_{ij}\bigg)\bigg( \int_K |\det R|^{\alpha+d} \prod_{i=1}^d R_{ii}^{-i}dR_{ii}\bigg) \\
&= \mu(O(d)) \bigg( \prod_{i<j} \int_{I_{ij}} dR_{ij}\bigg)\bigg( \prod_{i=1}^d \int_{I_{ii}} R_{ii}^{\alpha+d-i}dR_{ii}\bigg),
\end{align*}
where $\mu$ is proportional to normalized Haar measure on $O(d)$. The right side is finite, provided $\alpha>-1$. If on the other hand $\alpha\le-1$, take $I_{ij}=(0,1)$ for all $i\le j$, and set $A=O(d)\cdot K$. Then $A$ is relatively compact, but $\int_Aw(x)dx=\infty$.
\end{proof}

Consider the differential operator $\nabla^*$ given by
\[
\nabla^* G = - (\nabla + (\delta-1)x^{-\top})\bullet G, \qquad G \in C^1(E;\M^d).
\]
This notation is justified by the following integration by parts formula, which shows that $\nabla^*$ acts as an adjoint of $\nabla$. Together with the observation that $\Lcal = -\frac{1}{2}\nabla^*\nabla$, this will imply that $\Lcal$ is indeed a symmetric operator on $L^2(E,m)$.

\begin{theorem}[Integration by parts formula] \label{T:ibp}
Suppose $\delta>0$, and consider $f\in C^1_c(E)$ and $G\in C^1(E; \M^d)$. If $\delta\le 1$, assume that $G(x)$ is tangent to $\partial E$ at $x$ for all $x\in \partial E$. If $\delta<1$, assume in addition that $G(x)\bullet x^{-\top}$ is locally bounded. Then
\begin{equation}\label{eq:ibp}
\langle \nabla f, G\rangle  =  \langle f, \nabla^*G\rangle.
\end{equation}
\end{theorem}

\begin{proof}
See Appendix~\ref{A:ibp}.
\end{proof}

\begin{remark}
It is not hard to show that local boundedness of $G(x)\bullet x^{-\top}$ implies that $G(x)$ is tangent to $\partial E$ at $x$ for all $x\in \partial E$. Hence for $g\in\CNeu$, $G=\nabla g$ will always satisfy the assumptions of Theorem~\ref{T:ibp}.
\end{remark}

The fact that $\Lcal$ is symmetric is apparent from the equalities
\begin{equation} \label{eq:fLg}
\langle f, -\Lcal g\rangle = \frac{1}{2} \langle f, \nabla^*\nabla g\rangle = \frac{1}{2} \langle \nabla f, \nabla g \rangle,
\end{equation}
valid for any $f\in C^1_c(E)$ and any $g\in\CNeu$. If $\delta>1$ we may take any $g\in C^2_c(E)$. The BESM process is now defined as follows.

\begin{definition}[BESM semigroup]
A symmetric sub-Markovian strongly continuous contraction semigroup $(T_t:t\ge 0)$ on $L^2(E,m)$ is called a \emph{${\rm BESM}(\delta,d)$ semigroup} if its generator extends $(\Lcal, \CNeu)$.
\end{definition}

An $E$-valued Markov process $X$ is said to be \emph{$m$-symmetric} if its transition function $p_t(x,dy)$ is $m$-symmetric. In this case the operators $f\mapsto\int_E f(y)p_t(\cdot,dy)$, where $f$ is bounded and in $L^2(E,m)$, can be extended to all of $L^2(E,m)$, see \cite[page~30]{Fukushima:2011ys}. This extension is called the \emph{$L^2(E,m)$ semigroup} of~$X$.

\begin{definition}[BESM process]
An $E$-valued $m$-symmetric Markov process whose $L^2(E,m)$ semigroup is a ${\rm BESM}(\delta, d)$ semigroup is called a \emph{${\rm BESM}(\delta, d)$ process}.
\end{definition}

While uniqueness of the BESM semigroup and process is a delicate matter, existence is straightforward via the theory of Dirichlet forms. In view of \eqref{eq:fLg} it is natural to consider the symmetric bilinear form
\[
\Ecal(f,g) = \frac{1}{2} \langle \nabla f, \nabla g\rangle, \qquad f,g \in C^1_c(E).
\]
This form is closable on $L^2(E,m)$, as can be deduced from Theorem~\ref{T:ibp} as follows. Pick a sequence $(f_n)$ in $C^1_c(E)$ converging to zero in $L^2(E,m)$, such that $\lim_{n,k}\Ecal(f_n-f_k,f_n-f_k)=0$. We must show $\lim_n \Ecal(f_n,f_n)=0$. Since $(\nabla f_n)$ is a Cauchy sequence in $L^2(E,m;\M^d)$, it has a strong limit $F$. For any $G \in C^\infty_c(E;\M^d)$ vanishing on a neighborhood of $\partial E$, Theorem~\ref{T:ibp} yields
\[
\langle F,G\rangle = \lim_n\ \langle \nabla f_n,G\rangle = \lim_n \langle f_n, \nabla^*G\rangle = 0.
\]
It follows that $F=0$ $m$-a.e., establishing closability. We define
\[
(\Ecal, D(\Ecal)) = \text{closure of } (\Ecal,C^1_c(E)).
\]
An application of \cite[Theorem~3.1.2]{Fukushima:2011ys} then shows (after routine verification of the conditions of that theorem) that $(\Ecal, D(\Ecal))$ is a regular, strongly local Dirichlet form. Furthermore, \eqref{eq:fLg} implies that the generator associated with $\Ecal$ coincides with $\Lcal$ when acting on functions in $\CNeu$ (or in $C^2_c(E)$ when $\delta>1$.) With some abuse of notation, we therefore let
\[
(\Lcal, D(\Lcal)) = \text{generator of } (\Ecal, D(\Ecal)),
\]
noting that the domain $D(\Lcal)$ contains $\CNeu$, and even contains $C^2_c(E)$ if $\delta>1$. In particular, the semigroup $(T_t:t>0)$ on $L^2(E,m)$ associated with $\Ecal$ and $\Lcal$ is a ${\rm BESM}(\delta,d)$ semigroup.

A corresponding ${\rm BESM}(\delta,d)$ process is then obtained as the $m$-symmetric Hunt process $X$ on $E$ associated with the Dirichlet form $(\Ecal, D(\Ecal))$, see \cite[Theorems~7.2.1 and~7.2.2]{Fukushima:2011ys}. The strongly local property of $\Ecal$ implies that this process has continuous paths. However, it is not guaranteed \emph{a priori} that $X$ is conservative; we now prove that it is, thereby obtaining existence of the ${\rm BESM}(\delta,d)$ process. In the following, let $\P^x$ be the law of $X$ starting from $x\in E$.

\begin{proposition} \label{P:conservative}
The Dirichlet form $\Ecal$ is \emph{conservative}, i.e.~the semigroup $(T_t:t>0)$ satisfies $T_t1=1$ for all $t>0$. Consequently, $X$ can be chosen so that
\[
\P^x(X_t\in E {\rm\ for\ all\ } t\ge 0)=1 \quad \text{for all } x\in E.
\]
\end{proposition}

\begin{proof}
By \cite[Theorem~1.6.6]{Fukushima:2011ys}, $\Ecal$ is conservative if there is a sequence $(f_n)\subset D(\Ecal)$ such that $0\le f_n\le 1$ and $\lim_n f_n =1$ $m$-a.e., and such that
\[
\lim_n \Ecal(f_n,g)=0 \text{ holds for any } g \in D(\Ecal)\cap L^1(E,m).
\]
To construct such a sequence, pick $\phi\in C^2(\R)$ satisfying $\phi(t)=1$ for $t\le 0$, $\phi(t)=0$ for $t\ge 1$, and with $\phi'$, $\phi''$ uniformly bounded. For $n\ge 1$ define $f_n(x) = \phi(\| x \| - n)$. Differentiating twice yields
\begin{align*}
\nabla f_n(x) &= \phi'(\|x\|-n) \frac{x}{\|x\|}, \\
\Delta f_n(x) &= \phi''(\|x\|-n) + \phi'(\|x\|-n)\frac{d^2-1}{\|x\|}.
\end{align*}
Since $\nabla f_n$ and $\Delta f_n$ both vanish outside the set $E_n=\{x\in E:n\le\|x\|<n+1\}$, we obtain $f_n\in \CNeu\subset D(\Lcal)$ as well as
\[
|\Ecal(f_n,g)| =  \left| \int_{E_n} \Lcal f_n(x) g(x) m(dx) \right| \le \sup_{x\in E_n} |\Lcal f_n(x)| \int_{E_n} |g(x)| m(dx),
\]
where H\"older's inequality was applied. For $n\ge 1$, the supremum is bounded by a constant $c>0$ that is independent of $n$. Hence
\[
\sum_{n\ge 1} |\Ecal(f_n,g)| \le c \sum_{n\ge 0} \int_{E_n} |g(x)| m(dx) = c \| g \|_{L^1(E,m)} < \infty.
\]
We deduce that $\lim_n |\Ecal(f_n,g)|=0$, showing that $\Ecal$ is conservative. The statement about $X$ now follows from \cite[Exercise~4.5.1]{Fukushima:2011ys}.
\end{proof}

\section{Some properties and the relation to Wishart processes} \label{S:WIS}

Throughout this section $X$ denotes a BESM$(\delta,d)$ process with $\delta>0$, and $\P^x$ denotes its law when started from $x\in E$. Our goal is to study some of its basic properties, in particular the relation to Wishart processes. Much of the analysis relies on the following standard result, which states that $X$ solves the martingale problem for $\Lcal$. It can be deduced, for example, from \cite[Corollary~5.4.1 and Theorem~5.1.3]{Fukushima:2011ys}. Here and throughout this section, a property holds for {\em quasi-every (q.e.) $x\in E$} if it holds outside a set of zero capacity; see~\cite[page 68]{Fukushima:2011ys}.

\begin{lemma}\label{L:mg}
Pick any $f\in D(\Lcal)$ such that $\Lcal f$ is locally $m$-integrable on $E$. For q.e.~$x\in E$ we have $\int_0^t |\Lcal f(X_s)|ds<\infty$ for all $t\ge 0$, $\P^x$-a.s., and the process
\begin{equation}\label{eq:mg}
f(X_t) - f(x) - \int_0^t \Lcal f(X_s)ds, \qquad t\ge 0,
\end{equation}
is a square integrable martingale under $\P^x$.
\end{lemma}

\begin{remark} \label{R:1}
The exceptional set for which the conclusion of the lemma fails depends on the function $f$ in general. We conjecture that $X$ is in fact strongly Feller. In this case the quantifier ``for q.e.~$x\in E$'' can be replaced with ``for every $x\in E$'' in the above lemma, as well as in all subsequent results.
\end{remark}

\begin{proposition}
Suppose $\delta>1$. For q.e.~$x\in E$ we have
\begin{enumerate}
\item $\int_0^t \|X_s^{-\top}\|ds<\infty$ for all $t\ge 0$, $\P^x$-a.s.,
\item the process $W$ defined via
\begin{equation} \label{eq:SDE}
X_t = X_0 + W_t + \frac{\delta-1}{2} \int_0^t X_s^{-\top}ds, \qquad t\ge 0,
\end{equation}
is $\M^d$-valued Brownian motion under~$\P^x$.
\end{enumerate}
\end{proposition}

\begin{proof}
We saw in Section~\ref{S:DE} that $C^2_c(E)\subset D(\Lcal)$ if $\delta>1$, and Theorem~\ref{T:Radon} implies that $1/\det(x)$ is locally $m$-integrable in this case. Let $\Ccal$ be a countable subset of $C^2_c(E)$. Lemma~\ref{L:mg} then implies that there is an exceptional set $N\subset E$ such that for all $x\in E\setminus N$ we have (i), and \eqref{eq:mg} defines a $\P^x$~martingale for all $f\in\Ccal$. Choosing $\Ccal$ suitably, standard arguments (see for instance \cite[Theorem~V.20.1]{Rogers:2000ve}) show that $X$ solves the stochastic differential equation associated with $\Lcal$---that is, (ii)~holds.
\end{proof}

\begin{corollary} \label{C:semimg1}
For $\delta>1$, $X$ is a semimartingale under $\P^x$ for q.e.~$x\in E$.
\end{corollary}

We now describe the properties of the transformed processes $X^\top X$, $\|X\|$, $\det X$. The main observation is the following. Define the map
\[
\Phi:\M^d\to \S^d_+, \qquad x\mapsto x^\top x,
\]
and consider the operator
\[
\Lcal^{\rm WIS}g(z) = \tr(2z\nabla^2 + \alpha\nabla)g(z), \qquad z\in\S^d_+,\quad g\in C^\infty_c(\S^d_+).
\]
This is the generator of the ${\rm WIS}(\alpha,d)$ process, see~\cite{Bru:1991}. The Wishart process exists and is nondegenerate (in the sense of not being absorbed when it hits the boundary~$\partial\S^d_+$) precisely when
\begin{equation}\label{eq:alphaWIS}
\alpha>d-1.
\end{equation}
Now, for any $g\in C^\infty_c(\S^d_+)$ and any $x\in E$, one readily verifies the identities
\begin{equation} \label{eq:nablagPhi}
\nabla (g\circ\Phi)(x) = 2x\nabla g(x^\top x), \qquad \frac{1}{2}\Delta (g\circ\Phi)(x) = \tr(2 x^\top x \nabla^2 + d\nabla) g(x^\top x).
\end{equation}
Consequently we have
\begin{equation}\label{eq:WPhi}
g\circ\Phi \in D(\Lcal) \qquad \text{and} \qquad \Lcal (g\circ\Phi) = ( \Lcal^{\rm WIS}g ) \circ \Phi,
\end{equation}
where the latter function lies in $C^\infty_c(E)$ and in particular is locally $m$-integrable. An application of Lemma~\ref{L:mg} then shows that $\Phi(X)=X^\top X$ solves the martingale problem for $(\Lcal^{\rm WIS}, C^\infty_c(\S^d_+))$, with $\alpha=d-1+\delta$, and hence is a Wishart process. This proves part~(i) of the following theorem.

\begin{theorem}  \label{T:WIS}
The following statements hold.
\begin{enumerate}
\item For q.e.~$x\in E$, the law of $X^\top X$ under $\P^x$ is that of a {\rm WIS}$(\alpha,d)$ process starting from $x^\top x$, where $\alpha=d-1+\delta$.
\item Let $dz$ denote Lebesgue measure on $\S^d_+$. For $\alpha\in\R$, define a measure $\widehat m$ on $\S^d_+$ by
\begin{equation} \label{eq:T:WIS:1}
\widehat m(dz) = (\det z)^{(\alpha-d-1)/2}dz.
\end{equation}
Then $\widehat m$ is a Radon measure if and only if $\alpha>d-1$.
\item Let $\alpha=d-1+\delta$ and define $\widehat m$ by~\eqref{eq:T:WIS:1}. Define a symmetric bilinear form on $C^\infty_c(\S^d_+)$ by
\[
\Ecal^{\rm WIS}(g,h) = \frac{1}{2} \int_{\S^d_+} 4\, \tr\left(\nabla g(z)\, z\, \nabla h(z) \right) \widehat m(dz).
\]
This form is closable in $L^2(\S^d_+,\widehat m)$. Its closure is a regular, strongly local, conservative Dirichlet form, whose generator coincides with $\Lcal^{\rm WIS}$ on $C^\infty_c(\S^d_+)$. In particular, $(\Lcal^{\rm WIS},\,C^\infty_c(\S^d_+))$ is a symmetric operator on $L^2(\S^d_+,\widehat m)$.
\end{enumerate}
\end{theorem}

\begin{proof}
Part~(i) was proved above. For part~(ii), it follows from \cite[Theorem~2.1.14]{Muirhead:1982fk} that $\widehat m = c\, \Phi_*m$ for some constant $c>0$, where $\Phi_*m$ is the pushforward of~$m$ under~$\Phi$. Moreover, due to the bounds $\|x^\top x\| \le \|x\|^2\le \sqrt{d}\|x^\top x\|$, we have that $K\subset E$ is bounded if and only if $\Phi(K)\subset\S^d_+$ is bounded. The result now follows from Theorem~\ref{T:Radon}. It remains to prove part~(iii), and we start by expressing $\Ecal^{\rm WIS}$ in terms of $\Ecal$. For $g,h\in C^\infty_c(\S^d_+)$ we have
\begin{align*}
\Ecal(g\circ\Phi, h\circ\Phi)
&= \frac{1}{2} \int_E \nabla(g\circ\Phi)(x) \bullet \nabla(h\circ\Phi)(x)\, m(dx) \\
&= \frac{1}{2} \int_E 4\,\tr\left(\nabla g(x^\top x)\, x^\top x\, \nabla h(x^\top x)\right) m(dx) \\
&= \frac{1}{2} \int_{\S^d_+} 4\, \tr\left(\nabla g(z)\, z\, \nabla h(z) \right)\Phi_*m(dz) \\
&= c^{-1} \Ecal^{\rm WIS}(g,h),
\end{align*}
where we used the definition of $\Ecal$, the first identity in~\eqref{eq:nablagPhi}, the change of variable theorem, and finally the expression for $\widehat m$. This together with the equality
\[
\| g \|^2_{L^2(\S^d_+,\widehat m)} = c\| g\circ\Phi \|^2_{L^2(E,m)}
\]
lets us deduce closability, regularity and strong locality from the corresponding properties of~$\Ecal$. Conservativeness is proved as in Proposition~\ref{P:conservative} by observing that the functions $f_n$ appearing there are of the form $f_n=g_n\circ\Phi$.

To complete the proof we must relate $\Ecal^{\rm WIS}$ to the operator $\Lcal^{\rm WIS}$. Using~\eqref{eq:fLg} and~\eqref{eq:WPhi} we get, for $g,h\in C^\infty_c(\S^d_+)$,
\begin{align*}
\Ecal(g\circ\Phi, h\circ\Phi)
&= - \int_E g\circ\Phi(x)\, \Lcal(h\circ\Phi)(x)\, m(dx) \\
&= - \int_E g\circ\Phi(x)\,(\Lcal^{\rm WIS}h)\circ\Phi(x)\, m(dx) \\
&= - \int_{\S^d_+} g(z)\,\Lcal^{\rm WIS}h(z)\, \Phi_*m(dz).
\end{align*}
Combining this with the previous expression, we arrive at
\[
\Ecal^{\rm WIS}(g,h) =  - \int_{\S^d_+} g(z)\,\Lcal^{\rm WIS}h(z)\, \widehat m(dz),
\]
which implies that $\Lcal^{\rm WIS}$ coincides with the generator of $\Ecal^{\rm WIS}$ on $C^\infty_c(\S^d_+)$.
\end{proof}

\begin{remark}
It is interesting to note that the restriction $\delta>0$ corresponds exactly to the (well-known) condition~\eqref{eq:alphaWIS} for the existence of non-degenerate Wishart process, see \cite[Theorem~2]{Bru:1991}. Theorem~\ref{T:WIS} connects this to the Radon property of the symmetrizing measure~$\widehat m$.
\end{remark}

\begin{remark} \label{R:WISTD}
The transition density $q(t,u,z)$ of the ${\rm WIS}(\alpha,d)$ process is given in~\cite{Donati-Martin:2004vn}. In terms of the measure~$\widehat m$ it becomes
\[
q(t,u,z)dz = \frac{1}{(2t)^{\alpha d/2}\Gamma_d(\alpha/2)}\exp\left(-\frac{1}{2t}\tr(u+z)\right)\,_0F_1\left(\frac{\alpha}{2}; \frac{zu}{4t^2}\right) \widehat m(dz),
\]
where $\Gamma_d$ is the multivariate Gamma function, and $_0F_1$ is a hypergeometric function with matrix argument; see~\cite{Donati-Martin:2004vn} for the precise definitions. Note that the density with respect to $\widehat m$ is symmetric, as it should be.
\end{remark}

\begin{corollary} \label{C:Xprops}
Let $X$ be a ${\rm BESM}(\delta,d)$ process as above. The following statements hold for q.e.~$x\in E$.
\begin{enumerate}
\item $\|X\|$ is a ${\rm BES}(d\alpha)$ process under $\P^x$, where $\alpha=d-1+\delta$.
\item $\det X$ is a time-changed ${\rm BES}(\delta)$ process under $\P^x$. More specifically, define
\[
A_t=\int_0^t\|\adj X_s\|^2ds, \qquad \xi_u=\det X_{C(u)},
\]
where $C(u)=\inf\{t\ge 0: A_t>u\}$ is the right-continuous inverse of $A$. Then $A$ is strictly increasing, and $\xi$ is a ${\rm BES}(\delta)$ process stopped at $A_\infty$.
\end{enumerate}
\end{corollary}

\begin{proof}
Part~(i) is immediate from the well-known fact that the trace of a ${\rm WIS}(\alpha,d)$ process is a ${\rm BESQ}(d\alpha)$ process, see e.g.~\cite{Bru:1991}. We now prove part~(ii). Define $Z=X^\top X$, which is a ${\rm WIS}(\alpha,d)$ process with $\alpha=d-1+\delta$ by Theorem~\ref{T:WIS}. With $q(t,u,z)$ as in Remark~\ref{R:WISTD}, we have
\begin{align*}
\int_0^\infty \P^x(\rk X_t \le d-2) dt &= \int_0^\infty \P^x(\rk Z_t \le d-2) dt \\
&= \int_0^\infty \int_{\{u\in\S^d\,:\,\rk u\le d-2\}} q(t,x^\top x,z) dz\, dt =0.
\end{align*}
Thus $\{t:\adj X_t=0\}$ is a nullset, whence $A$ is strictly increasing. Next, $\det Z$ satisfies
\[
\det Z_t = \det Z_0 + 2\int_0^t \sqrt{\det Z_s} \sqrt{ \tr( \adj Z_s)} d\beta_s + \delta \int_0^t \tr( \adj Z_s) ds
\]
for some standard Brownian motion $\beta$; see~\cite[Section~4]{Bru:1991}. Hence after a time change (see \cite[Proposition~V.1.4]{Revuz:1999zr}) and using that $A_{C(u)}=u\wedge A_\infty$, we obtain
\[
\det Z_{C(u)} = \det Z_0 + 2\int_0^{u\wedge A_\infty} \sqrt{\det Z_{C(v)}} d\widetilde\beta_v  + \delta (u\wedge A_\infty),
\]
where we defined
\[
\widetilde \beta_u = \int_0^{u\wedge A_\infty} \sqrt{ \tr( \adj Z_{C(v)})} d\beta_{C(v)} = \int_0^{C(u)} \sqrt{ \tr( \adj Z_s)} d\beta_s,
\]
which is Brownian motion stopped at $A_\infty$. It follows that $\det Z_{C(\cdot)}$ satisfies the stochastic differential equation for the BESQ$(\delta)$ process, stopped at $A_\infty$. Since $\det X_t = \sqrt{\det Z_t}$ the result follows.
\end{proof}

\begin{corollary} \label{C:semimg2}
For $0<\delta<1$, $X$ fails to be a semimartingale under $\P^x$ for q.e.~$x\in E$.
\end{corollary}

\begin{proof}
If $X$ were a semimartingale, then so would $\det X$, as well as the process $\xi$ in Corollary~\ref{C:Xprops}. However, the ${\rm BES}(\delta)$ process, $0<\delta<1$, fails to be a semimartingale on any interval larger than $[0,\tau)$, where $\tau$ is the first time it hits zero. It thus suffices to show that $\det X_t=0$ for some finite~$t$. Indeed, setting $u=A_t$ then yields $\xi_u=\det X_t=0$ and $u<A_\infty$, due to the strict increase of $A$. Thus $\xi$ hits zero before it is stopped, and fails to be a semimartingale. This contradiction shows that $X$ could not have been a semimartingale. The fact that $\det X_t=0$ for some finite $t$ follows from the corresponding well-known fact for the Wishart process.
\end{proof}

\begin{remark}
In view of Corollaries~\ref{C:semimg1} and~\ref{C:semimg2} one wonders whether $X$ is a semimartingale for $\delta=1$. Just as in the scalar case it turns out that it is---in fact, $X$ is reflected Brownian motion. We do not discuss this further here.
\end{remark}

We close this section with a pathwise construction of the BESM process as the strong solution to the stochastic differential equation~\eqref{eq:SDE}. This construction only works for~$\delta\ge 2$ and if the process starts from the interior of~$E$. Whether strong solutions exist for all $\delta>1$ is an open question also in the case of Wishart processes.

\begin{proposition} \label{P:str}
Suppose $\delta\ge 2$, and let $W$ be standard $\M^d$-valued Brownian motion defined on some probability space. The stochastic differential equation
\begin{equation} \label{eq:SDEstrong}
dX_t = dW_t + \frac{\delta-1}{2}X^{-\top}_s ds, \qquad X_0=x,
\end{equation}
has a unique $E^o$-valued strong solution for every $x\in E^o$.
\end{proposition}

\begin{proof}
The proof uses the so-called {\em McKean's argument}; see \cite[Section~4.1]{Mayerhofer:2011kx} for a thorough treatment in a related setting. We only sketch the proof here. Since $x\mapsto x^{-\top}$ is locally Lipschitz on $E^o$, a strong solution $X$ exists for $t<\zeta$, where $\zeta=\lim_{n\to\infty} \inf\{t\ge 0: \det X_t<n^{-1} \text{ or } \|X_t\| > n\}$. We claim that $\zeta=\infty$. To see this, set $Z=X^\top X$ and note that we have
\[
Z_t = \int_0^t X_s dW_s + \int_0^t dW_s^\top X_s^\top + (d-1+\delta)\Id\, t, \qquad t<\zeta,
\]
where we used the equality $dW_tdW_t = d\,\Id\, dt$. Defining $\widetilde W_t=\int_0^t (X_s^\top X_s)^{-1/2}X_sdW_s$, $t<\zeta$, we have
\[
Z_t = \int_0^t \sqrt{Z_s}d\widetilde W_s + \int_0^t d\widetilde W_s^\top \sqrt{Z_s} + (d-1+\delta)\Id t, \qquad t<\zeta,
\]
and after verifying that $\widetilde W$ is again $\M^d$-valued Brownian motion on $[0,\zeta)$, it follows that $Z$ is a ${\rm WIS}(d-1+\delta,d)$ process on $[0,\zeta)$ (this calculation is of course closely related to the one leading up to~\eqref{eq:WPhi}.) Since $\delta\ge 2$, well-known properties of the Wishart process (c.f.~\cite{Bru:1991}, or Corollary~\ref{C:Xprops}(ii) above) imply that $\det X_t=\sqrt{\det Z_t}$ stays strictly positive, and that $\|X_t\|^2=\tr Z_t$ is nonexplosive. Hence $\zeta=\infty$ as claimed, and the result follows.
\end{proof}

\section{Uniqueness} \label{S:MU}

The goal of this section is to establish uniqueness of the BESM semigroup. Specifically, we will prove that $(\Lcal,\CNeu)$ is \emph{Markov unique}. This means that there is at most one (and hence exactly one) symmetric sub-Markovian strongly continuous contraction semigroup on $L^2(E,m)$ whose generator extends $(\Lcal,\CNeu$), see \cite[Definition~1.1.2]{Eberle:1999uq}. Since the Hunt process corresponding to such a semigroup is unique up to equivalence, this form of uniqueness will hold for any realization of the BESM process as a Hunt process. In particular, uniqueness in law is guaranteed. (Two symmetric Hunt processes are called equivalent if their transition functions coincide outside a \emph{properly exceptional set}, see Section~4.1 in~\cite{Fukushima:2011ys}.)

Note that $m(\partial E)=0$. Therefore $L^2(E,m)$ and $L^2(E^o,m)$ can be identified, implying that it is enough to prove Markov uniqueness of $(\Lcal,\CNeu)$ as an operator on the latter space. To do this we will apply a general result by Eberle \cite[Corollary~3.2]{Eberle:1999uq} that relies on studying the relationship between various weighted Sobolev spaces, which we now introduce. To simplify notation we henceforth write
\[
\Omega = E^o = \{x\in\M^d: \det x>0\}.
\]
Observe that $1/(\det x)^{\delta-1}$ is locally integrable on $\Omega$. Hence by \cite[Theorem~1.5]{Kufner/Opic:1984}, $L^2(\Omega ,m)$ is continuously imbedded in $L^1_{\rm loc}(\Omega)$. In particular, every $f\in L^2(\Omega ,m)$ has a gradient $Df$ in the sense of distributions, and one can define the \emph{weak Sobolev space}
\[
W^{1,2}(\Omega ,m) = \left\{ f\in L^2(\Omega ,m) :  Df \in L^2(\Omega ,m; \M^d) \right\}.
\]
Equipped with the norm
\[
\|f\|_{W^{1,2}(\Omega ,m)} = \left( \int_\Omega  |f(x)|^2m(dx) + \int_\Omega  \|Df(x)\|^2m(dx) \right)^{1/2},
\]
$W^{1,2}(\Omega ,m)$ becomes a Hilbert space \cite[Theorem~1.11]{Kufner/Opic:1984}. Appendix~\ref{A:W12} reviews some basic properties of $W^{1,2}(\Omega ,m)$ that will be needed in the sequel. We also consider the following \emph{strong Sobolev spaces} (here the word \emph{completion} is always meant with respect to the norm $\|\cdot\|_{W^{1,2}(\Omega ,m)}$):
\begin{align*}
H^{1,2}(\Omega ,m) &= \text{completion of } C^1(\Omega ) \\[2mm]
\HNeu(\Omega ,m) &= \text{completion of } \CNeu \\[2mm]
H^{1,2}_0(\Omega ,m) &= \text{completion of } C^1_c(\Omega ).
\end{align*}
These are all Hilbert spaces by construction, and we automatically have
\begin{equation} \label{eq:incl111}
H^{1,2}_0(\Omega ,m) \subset \HNeu(\Omega ,m) \subset H^{1,2}(\Omega ,m) \subset W^{1,2}(\Omega ,m).
\end{equation}
The main result of this section shows that for any $\delta>0$, the last two inclusions are equalities; and if $\delta\ge 2$, all three inclusions are equalities. This will lead to Markov uniqueness of the BESM semigroup.

\begin{theorem} \label{T:MU}
The following statements hold.
\begin{enumerate}
\item\label{T:MU:1} If $0<\delta<2$, then $\HNeu(\Omega ,m)=W^{1,2}(\Omega ,m)$.
\item\label{T:MU:2} If $\delta \ge 2$, then $H^{1,2}_0(\Omega ,m)=W^{1,2}(\Omega ,m)$.
\end{enumerate}
\end{theorem}

Before giving the proof of Theorem~\ref{T:MU}, we note that Markov uniqueness now follows directly from the basic criterion for Markov uniqueness given in~\cite[Corollary~3.2]{Eberle:1999uq}, which only relies on the equality $W^{1,2}(\Omega,m)=\HNeu(\Omega,m)$.

\begin{corollary}
For any $\delta>0$, $(\Lcal,\CNeu)$ is Markov unique.
\end{corollary}

Next, since every $f\in C^1_c(E)$ satisfies $f|_\Omega\in C^1(\Omega)$, we have $D(\Ecal)\subset H^{1,2}(\Omega,m)$. Furthermore, since $\CNeu\subset C^1_c(E)$ holds we have $\HNeu(\Omega,m)\subset D(\Ecal)$, and we deduce the following corollary of Theorem~\ref{T:MU}:

\begin{corollary}
For any $\delta>0$, $D(\Ecal)=W^{1,2}(\Omega,m)$.
\end{corollary}

\begin{remark}
An application of Theorem~\ref{T:Ap} with $\alpha=\delta-1$ shows that $|\det x|^{\delta-1}$ is an $A_2$-weight if $0<\delta<2$. In particular, \cite[Theorem~2.5]{Kilpelainen:1994} then implies that the last inclusion in \eqref{eq:incl111} is in fact an equality.
\end{remark}

The proof of Theorem~\ref{T:MU} relies crucially on Theorems~\ref{T:cap0} and~\ref{T:Ap}. It also uses what we refer to as {\em tube segments}, discussed in Appendix~\ref{A:tube}, as well as some basic properties of the space $W^{1,2}(\Omega,m)$, reviewed in Appendix~\ref{A:W12}. Most of the difficulties arise for $0<\delta<2$. In fact, the case $\delta\ge 2$ only requires (the second part of) the following lemma, which uses Theorem~\ref{T:cap0} but not Theorem~\ref{T:Ap}. Henceforth, for any subset $\Gamma\subset\M^d$ we define
\[
W_\Gamma=\left\{ h\in W^{1,2}(\Omega,m)\ : \
\begin{array}{l}
\text{$h$ is bounded with compact support, }  \\
\text{$h=0$ on $U\cap\Omega$ for some open } U\subset\M^d \text{ with } \Gamma\subset U
\end{array}
\right\}.
\]

\begin{lemma} \label{L:dense1}
\begin{enumerate}
\item Let $\delta>0$ and define $\Gamma=\bigcup_{k\le d-2}M_k$. Then $W_\Gamma$ is dense in $W^{1,2}(\Omega,m)$.
\item Let $\delta\ge2$ and define $\Gamma=\bigcup_{k\le d-1}M_k=\partial E$. Then $W_\Gamma$ is dense in $W^{1,2}(\Omega,m)$.
\end{enumerate}
\end{lemma}

\begin{proof}
Since the elements $f\in W^{1,2}(\Omega,m)$ that are bounded with bounded support are dense (see Lemma~\ref{L:W12props}(iv)), it suffices to approximate such $f$ by elements $h\in W_\Gamma$. By scaling we may assume $|f|\le 1$. Let $\varepsilon>0$ be arbitrary, and let $C>0$ be the constant given by Lemma~\ref{L:W12fg}. By Theorem~\ref{T:cap0} there is a neighborhood $U$ of $\Gamma$ and an element $g\in W^{1,2}(\M^d,m)$ such that $g\ge1$ on $U$ and $\|g\|_{W^{1,2}(\Omega,m)}^2\le \varepsilon/C$. By truncating (using Lemma~\ref{L:W12props}(ii)) we may assume $|g|\le 1$. Define $h=(1-g)f$. Lemma~\ref{L:W12fg} then yields
\[
\|f-h\|^2_{W^{1,2}(\Omega,m)} = \|fg\|^2_{W^{1,2}(\Omega,m)} \le C  \frac{\varepsilon}{C} + \varepsilon = 2\varepsilon.
\]
This proves the lemma.
\end{proof}

\begin{proof}[Proof of Theorem~\ref{T:MU}\ref{T:MU:2}]
In view of \eqref{eq:incl111} it is clear that it suffices to prove $W^{1,2}(\Omega,m)\subset H^{1,2}_0(\Omega,m)$, so we pick $f\in W^{1,2}(\Omega,m)$. By Lemma~\ref{L:dense1}, $f$ can be approximated by some bounded $h\in W^{1,2}(\Omega,m)$ whose support is bounded and bounded away from $\partial E$. By mollification (see Lemma~\ref{L:W12props}(i)), $h$ can in turn be approximated by some $g\in C^1_c(\Omega)$. Since $H^{1,2}_0(\Omega,m)$ is the completion of the set of such functions, we obtain $f\in H^{1,2}_0(\Omega,m)$, as desired.
\end{proof}

For $0<\delta<2$ the boundary no longer has zero capacity, which makes this case more delicate. The following lemma uses a powerful extension theorem due to Chua~\cite{Chua:1992} for weighted Sobolev spaces with Muckenhoupt weights. In particular, therefore, we will rely on the Muckenhoupt $A_2$ property of $|\det x|^{\delta-1}$ for $0<\delta<2$, which is asserted by Theorem~\ref{T:Ap}. Chua's theorem requires the domain to be a so-called {\em $(\varepsilon,\delta)$-domain}. Unfortunately, it does not appear straightforward to show that $\Omega$ itself is of this type. Instead we employ a partition of unity argument with an open cover consisting of tube segments; see Appendix~\ref{A:tube}. The intersection of $\Omega$ with a tube segment around some $x\in M_{d-1}$ is an $(\varepsilon,\delta)$-domain (indeed, a Lipschitz domain), and Chua's theorem becomes applicable.

\begin{lemma}  \label{L:dense2}
Let $0<\delta<2$ and define $\Gamma=\bigcup_{k\le d-2}M_k$. For any $f\in W_\Gamma$ there exists $g\in W^{1,2}(\M^d,m)$ satisfying $g=0$ on a neighborhood of~$\Gamma$ and $g=f$ on $\Omega$.
\end{lemma}

\begin{proof}
For each $x\in M_{d-1}$, let $U_x$ be a tube segment around~$x$, see Definition~\ref{D:tube} and Proposition~\ref{P:tube} in Appendix~\ref{A:tube}. Then $U_x\cap\Omega$ is diffeomorphic to $A\times(-1,1)\subset\R^{d-1}\times\R$, where $A$ is an open ball in $\R^{d-1}$. Hence $U_x\cap\Omega$ is a Lipschitz domain, and thus an $(\varepsilon,\delta)$-domain, see~\cite[p.~73]{Jones:1981}. Moreover, $f\in W^{1,2}(U_x\cap\Omega,m)$. Since $|\det x|^{\delta-1}$ is a Muckenhoupt $A_2$ weight by Theorem~\ref{T:Ap}, the extension theorem of Chua \cite[Theorem~1.1]{Chua:1992} yields an element
\[
\text{$f_x\in W^{1,2}(\M^d,m)$ with $f_x=f$ on $U_x\cap\Omega$.}
\]
Now, since $f\in W_\Gamma$, there is a compact set $K\subset\M^d$, bounded away from $\Gamma$, with $f=0$ on $\Omega\setminus K$. By compactness we can choose finitely many of the $U_x$, say $U_1,\ldots,U_n$ (and corresponding $f_1,\ldots,f_n$), as well some open $U_0\subset\M^d$ with $\overline U_0\subset\Omega$, such that $K\subset\bigcup_{i=0}^n U_i$. Defining $f_0=\psi f$, where $\psi$ is a smooth cutoff function with $\psi=1$ on $U_0$ and $\psi=0$ outside some $V_0\Subset\Omega$ with $U_0\Subset V_0$, we have $f_0\in W^{1,2}(\M^d,m)$ with $f_0=f$ on $U_0$. We then let $\{\psi_0,\ldots,\psi_n\}$ be a smooth partition of unity subordinate to $U_0,\ldots,U_n$, and define
\[
g = \sum_{i=0}^n \psi_i f_i \in W^{1,2}(\M^d,m).
\]
Since $x\in U_i$ holds whenever $\psi_i(x)\ne0$, and since $f_i=f$ on $U_i\cap\Omega$, we get $g=\sum_{i=0}^n \psi_i f= f$ on $\Omega$, as desired.
\end{proof}

Using Lemma~\ref{L:dense2}, one can replace mollification of elements in $W^{1,2}(\Omega,m)$ by mollification of elements in $W^{1,2}(\M^d,m)$. The latter is straightforward, while the former is not. The upshot is the following result.

\begin{lemma}  \label{L:dense3}
Let $0<\delta<2$ and define $\Gamma=\bigcup_{k\le d-2}M_k$. Then $C^\infty_c(E\setminus\Gamma)$ is dense in $W^{1,2}(\Omega,m)$.
\end{lemma}

\begin{proof}
By Lemma~\ref{L:dense1} it suffices to approximate elements $f\in W_\Gamma$. By Lemma~\ref{L:dense2} we then have $f=g|_\Omega$ for some $g\in W^{1,2}(\M^d,m)$ with $g=0$ near~$\Gamma$. By mollification, see \cite[Lemma~1.5]{Kilpelainen:1994}, we obtain $h\in C^\infty(\M^d)$ that approximates $g$ in $W^{1,2}(\M^d,m)$-norm. Choosing the support of the mollifier sufficiently small, we still have $h=0$ on a neighborhood of~$\Gamma$. Then $h|_{\Omega}\in C^\infty_c(E\setminus\Gamma)$ approximates~$f$.
\end{proof}

\begin{proof}[Proof of Theorem~\ref{T:MU}\ref{T:MU:1}]
In view of \eqref{eq:incl111} and the definition of $\HNeu(\Omega,m)$, we need to prove that $\CNeu$ is dense in $W^{1,2}(\Omega,m)$. By Lemma~\ref{L:dense3} it suffices to approximate elements $f\in C^\infty_c(E\setminus\Gamma)$, where $\Gamma=\bigcup_{k\le d-2}M_k$. An approximating function $h\in\CNeu$ can be constructed explicitly, relying on the fact that $f=0$ on a neighborhood of~$\Gamma$. We now give the details.

For $\varepsilon > 0$, let $\phi_\varepsilon\in C^2(\R_+)$ satisfy the following properties:
\begin{enumerate}
\item $\phi_\varepsilon(t) = 0$ for $t \geq \varepsilon$,
\item $|\phi_\varepsilon(t)|\leq \varepsilon$ and $|\phi'_\varepsilon(t)| \leq 3$ for all $t> 0$.
\item $\frac{1-\phi'_\varepsilon(t)}{t}$ and $\frac{\phi_\varepsilon(t)}{t}$ are bounded in $t$ (where the bound may depend on $\varepsilon$),
\end{enumerate}
Such $\phi_\varepsilon$ exists: first set $\phi_1(t)=t\psi(t)$ where $\psi$ is some smooth cutoff function, and then $\phi_\varepsilon(t)=\varepsilon \phi_1(t/\varepsilon)$. Now, let $K$ denote the support of $f$ and define
\[
g=\frac{\nabla f \bullet \nabla\det}{\|\nabla\det\|^2}\oo_K,
\qquad G = \nabla f - g\,\nabla\det.
\]
Note that $\nabla\det(x)=\adj x^\top\ne 0$ for all $x\in K \subset \M^d\setminus \Gamma$, so that $g$ is well-defined and in $C^2_c(E)$. Moreover, we have $G \bullet \nabla\det = 0$. Consider the function
\begin{equation} \label{eq:MU1}
h_\varepsilon = f - g\, \phi_\varepsilon\circ\det.
\end{equation}
We claim that $h_\varepsilon\in\CNeu$ and $h_\varepsilon \to f$ in $W^{1,2}(\Omega,m)$ as $\varepsilon\downarrow 0$. To prove this we first obtain, via a calculation using the chain and product rules, the following two expressions for~$\nabla h_\varepsilon$:
\begin{align}
\nabla h_\varepsilon &= \nabla f - (\phi'_\varepsilon\circ\det)\, g\, \nabla\det - (\phi_\varepsilon\circ\det)\nabla g
\label{eq:MU2} \\[2mm]
&= (1 - \phi'_\varepsilon\circ\det)\nabla f + (\phi'_\varepsilon\circ\det) G - (\phi_\varepsilon\circ\det)\nabla g.
\label{eq:MU3}
\end{align}
Equations \eqref{eq:MU1} and \eqref{eq:MU2} and properties~(i) and~(ii) of $\phi_\varepsilon$ yield the pointwise inequalities
\[
|f - h_\varepsilon| \le \varepsilon |g|
\]
and
\[
\|\nabla f - \nabla h_\varepsilon\| \le \Big(3 |g|\, \|\nabla\det\| + \|\nabla g\| \Big) {\bf 1}_{[0,\varepsilon]}\circ \det.
\]
Together with the fact that $m(\{x\in K: 0\le\det \le \varepsilon\})$ tends to zero as $\varepsilon\downarrow 0$, this yields $h_\varepsilon\in\CNeu$ and $h_\varepsilon \to f$ in $W^{1,2}(\Omega,m)$. It remains to check $h_\varepsilon \in \CNeu$. Clearly $h_\varepsilon\in C^2_c(\Omega)$. Moreover, \eqref{eq:MU3} together with the orthogonality $G\bullet\nabla\det=0$, as well as the fact that $x^{-\top}=\nabla\det(x) / \det(x)$, yield
\[
x^{-\top} \bullet \nabla h_\varepsilon = \frac{1-\phi_\varepsilon'\circ\det}{\phi_\varepsilon\circ\det}\, \nabla\det\,\bullet\,\nabla f  - \frac{\phi_\varepsilon\circ\det}{\det}\,\nabla\det\,\bullet\,\nabla g.
\]
Property~(iii) of $\phi_\varepsilon$ implies that the right side is bounded, as required. This completes the proof.
\end{proof}

\begin{remark}
Using results in \cite{Drusvyatskiy/Larsson:2015}, the space $\CNeu$ can be shown to be dense in $C^2_c(E)$ with respect to the norm $\|\cdot\|_{W^{1,2}(\Omega,m)}$. An alternative approach to proving Theorem~\ref{T:MU}\ref{T:MU:1} would therefore be to show directly that $C^2_c(E)$ is dense in $W^{1,2}(\Omega,m)$, for example by showing that $\Omega$ is an $(\varepsilon,\delta)$-domain and then apply Chua's extension theorem. Proving the $(\varepsilon,\delta)$~property does not appear to be straightforward---one obstruction is that $\Omega$ does not lie on one side of its boundary, as discussed in the Introduction.
\end{remark}

\section{Low-rank matrices have zero capacity} \label{S:cap0}

This section is devoted to proving that the sets $M_k$ consisting of rank~$k$ matrices have zero capacity for all sufficiently small~$k$. This is a key ingredient in the proof of Theorem~\ref{T:MU}, and also interesting in its own right. We use the following notion of capacity. For any subset $A\subset \M^d$, define
\[
{\rm Cap}(A) = \inf_f \int_{\M^d} \left( |f(x)|^2  + \|D f(x)\|^2\right) m(dx),
\]
where the infimum is taken over all $f\in W^{1,2}(\M^d,m)$ with $f\ge 1$ on an open neighborhood of $A$. The main result is the following.

\begin{theorem} \label{T:cap0}
Let $\delta>0$. For $k\in\{0,\ldots,d-2\}$, we have ${\rm Cap}(M_k)=0$. If $\delta\ge 2$, the same thing holds also for $k=d-1$.
\end{theorem}

\begin{remark}
The above definition of capacity differs from the {\em $(1,m)$-Sobolev capacity} in \cite[Definition~2.35]{Heinonen:2006fk}, where $W^{1,2}(\M^d,m)$ is replaced by $H^{1,2}(\M^d,m)$. It also differs from the {\em $1$-capacity} in \cite[Eqs.~(2.1.1)--(2.1.3)]{Fukushima:2011ys}, where $W^{1,2}(\M^d,m)$ is replaced by $D(\Ecal)$. However, Theorem~\ref{T:MU} and its corollaries imply that any $A\subset E$ with ${\rm Cap}(A)=0$ also has zero capacity in all the above senses.
\end{remark}

The core of the proof of Theorem~\ref{T:cap0} is an application of the following lemma, which bounds the growth of the determinant function near a point $x\in M_k$.

\begin{lemma} \label{L:detlr}
Let $k\le d-1$. There is a locally Lipschitz function $c_k:M_k\to\R_+$ such that
\[
|\det(x+v)| \le c_k(x)\|v\|^{d-k}, \qquad x\in M_k, \ v\in \M^d, \ \|v\|\le 1.
\]
\end{lemma}

\begin{proof}
By \cite[Corollary~5]{Bhatia/Jain:2009},
\[
|\det(x+v)-\det(x)| \le \sum_{i=1}^d p_{d-i}(\sigma_1(x), \ldots, \sigma_d(x)) \|v\|^i,
\]
where $\sigma(x)=(\sigma_1(x),\ldots,\sigma_d(x))$ is the vector of singular values of $x$, and $p_i$ is the $i$:th elementary symmetric polynomial in $d$ variables. Now, $p_{d-i}(\sigma_1(x),\ldots,\sigma_d(x))$ consists of a sum of terms, each of which is the product of $d-i$ distinct elements of $\sigma(x)$. However, since $\rk x=k$, only $k$ of those elements are nonzero. Therefore the product must contain at least one zero factor whenever $d-i>k$, implying that $p_{d-i}(\sigma_1(x),\ldots,\sigma_d(x))=0$ for these~$i$. Since in addition $\|v\|\le 1$ and $\det x=0$, we get
\[
|\det(x+v)| \le \|v\|^{d-k} \sum_{i=d-k}^d p_{d-i}(\sigma_1(x), \ldots, \sigma_d(x)).
\]
The local Lipschitz property follows from the smoothness of $p_{d-i}$ and the fact that the singular value map is Lipschitz continuous, see~\cite[Theorem~7.4.51]{Horn/Johnson:1985}.
\end{proof}

In proving Theorem~\ref{T:cap0}, the case $\delta=2$, $k=d-1$, turns out to require separate treatment using the following lemma.

\begin{lemma} \label{L:d2}
For each $\varepsilon<1$ there is a Lipschitz function $\phi_\varepsilon:\R_+ \to \R$ such that $0\le \phi_\varepsilon \le 1$, $\phi_\varepsilon=0$ on $[\varepsilon,\infty)$, $\phi_\varepsilon=1$ on a neighborhood of zero, and
\begin{equation} \label{eq:Ld2}
\lim_{\varepsilon\downarrow 0} \int_{\R_+} |\phi_\varepsilon'(t)|^2\, t\, dt = 0.
\end{equation}
\end{lemma}

\begin{proof}
Define functions $g_\varepsilon$ and $h_\varepsilon$ on $\R_+$ by
\[
g_\varepsilon(t) = \left(1-\left(\frac{t}{\varepsilon}\right)^\varepsilon\right)_+ \quad\text{and}\quad
h_\varepsilon(t) = \left\{
\begin{array}{ll}
(t/\varepsilon)^\varepsilon		& t \in [0, \varepsilon^{1+1/\varepsilon}) \\[2mm]
2\varepsilon - \varepsilon^{-1/\varepsilon}t &    t \in [\varepsilon^{1+1/\varepsilon}, 2 \varepsilon^{1+1/\varepsilon})\\[2mm]
0						& t\in [2\varepsilon^{1+1/\varepsilon},\infty)
\end{array}
\right.
\]
We claim that the function $\phi_\varepsilon=g_\varepsilon+h_\varepsilon$ has the stated properties. It is not hard to check that $0\le \phi_\varepsilon\le 1$ and that $\phi_\varepsilon$ equals zero on $[\varepsilon,\infty)$ and one on $[0, \varepsilon^{1+1/\varepsilon})$. The Lipschitz property then follows easily. It remains to verify \eqref{eq:Ld2}. First, note that
\[
\int_{\R_+} |g_\varepsilon'(t)|^2\, t\, dt = \varepsilon^{2-2\varepsilon} \int_0^\varepsilon t^{2\varepsilon-1}dt = \frac{\varepsilon}{2} \to 0 \quad (\varepsilon\downarrow 0).
\]
Moreover, since $|h_\varepsilon'| = |g_\varepsilon'|$ on $[0,\varepsilon^{1+1/\varepsilon})$, and since
\[
\int_{\varepsilon^{1+1/\varepsilon}}^{2\varepsilon^{1+1/\varepsilon}}|h_\varepsilon'(t)|^2\, t\, dt = \varepsilon^{-2/\varepsilon} \frac{(2\varepsilon^{1+1/\varepsilon})^2-(\varepsilon^{1+1/\varepsilon})^2}{2} = \frac{3}{2}\varepsilon^2 \to 0 \quad (\varepsilon\downarrow 0),
\]
it follows that $\lim_{\varepsilon\downarrow 0} \int_{\R_+}|h_\varepsilon'(t)|^2\, t\, dt = 0$. We now deduce~\eqref{eq:Ld2}.
\end{proof}

We are now ready to prove Theorem~\ref{T:cap0}. The proof uses the tube segments discussed in Appendix~\ref{A:tube}.

\begin{proof}[Proof of Theorem~\ref{T:cap0}]
Suppose for any fixed $\overline x\in M_k$ we can find a bounded neighborhood $U$ of $\overline x$ in $\M^d$ and bounded functions $g_\varepsilon\in W^{1,2}(U,m)$ such that each $g_\varepsilon$ equals one on a neighborhood of $M_k\cap U$, and $\lim_{\varepsilon\downarrow0}\|g_\varepsilon\|_{W^{1,2}(U,m)}=0$ holds. We then take an open set $V\subset\M^d$ with $\overline V\subset U$, and a smooth cutoff function $\phi\in C_c^\infty(\M^d)$ with $\phi=1$ on $V$ and $\phi=0$ on $\M^d\setminus U$. The function $f_\varepsilon=\phi g_\varepsilon$ then lies in $W^{1,2}(\M^d,m)$, is equal to one on a neighborhood of $M_k\cap V$, and satisfies $\lim_{\varepsilon\downarrow0}\|f_\varepsilon\|_{W^{1,2}(\M^d,m)}=0$ by Lemma~\ref{L:W12UVfg}. It follows that ${\rm Cap}(M_k\cap V)=0$. Since $M_k$ can be covered by countably many such sets $M_k\cap V$, we deduce ${\rm Cap}(M_k)=0$, as desired.

We thus focus on finding functions $g_\varepsilon$ as above. To this end, set $M=M_k$, $n_1=d^2-(d-k)^2$, $n_2=(d-k)^2$, pick $\overline x\in M$, and let $U$ be a tube segment around~$\overline x$, see Definition~\ref{D:tube} and Proposition~\ref{P:tube} in Appendix~\ref{A:tube}. Let
\[
\Phi:A\times B_1\to U
\]
be the corresponding diffeomorphism, where $A\subset\R^{n_1}$ is an open ball, and $B_1\subset\R^{n_2}$ is the open unit ball. Let $\pi:\R^{n_1}\times\R^{n_2}\to\{0\}\times\R^{n_2}$ be the projection onto the last $n_2$ coordinates. Let $\phi\in C^\infty(\R_+)$ be a cutoff function valued in $[0,1]$, equal to one on $[0,1/2]$, equal to zero on $[1,\infty)$, and with $|\phi'(t)|\le 3$ for all $t\in\R_+$. For each $0<\varepsilon<1$, define a map
\[
g_\varepsilon\ :\ U \to \R, \qquad g_\varepsilon\ =\ \phi(\varepsilon^{-1}\|\pi\circ\Phi^{-1}\|).
\]
We then have $g_\varepsilon\in C^\infty(U)$ and $g_\varepsilon=1$ on $\Phi(A\times B_{\varepsilon/2})$, a neighborhood of $M\cap U$. It remains to prove $g_\varepsilon\in W^{1,2}(U,m)$ and $\lim_{\varepsilon\downarrow0}\|g_\varepsilon\|_{W^{1,2}(U,m)}=0$. A computation based on the chain rule gives the gradient of $g_\varepsilon$,
\begin{align*}
\nabla g_\varepsilon
\ =\ \nabla (\Phi^{-1})\, \nabla\pi\circ\Phi^{-1}\, \frac{1}{\varepsilon}\phi'(\varepsilon^{-1}\|\pi\circ\Phi^{-1}\|) \frac{\pi}{\|\pi\|}\circ\Phi^{-1},
\end{align*}
where $\nabla(\Phi^{-1})$ denotes the transpose of the Jacobian matrix of $\Phi^{-1}$, and similarly for $\nabla\pi$. Hence
\begin{equation} \label{eq:Tcap01100}
\|\nabla g_\varepsilon\| \ \le\ \frac{3}{\varepsilon} \|\nabla (\Phi^{-1})\|_{\rm op}\, \|\nabla\pi\|_{\rm op} \ \le\ \frac{3C}{\varepsilon},
\end{equation}
where $C=\sup_{x\in U}\|\nabla(\Phi^{-1})(x)\|_{\rm op}$ is finite by property~(iv) of Definition~\ref{D:tube}, and where we used that the projection~$\pi$ is $1$-Lipschitz. Write
\[
U_\varepsilon=\Phi(A\times B_\varepsilon).
\]
We then have $g_\varepsilon=0$ on $U\setminus U_\varepsilon$, which yields
\[
\|g_\varepsilon\|^2_{W^{1,2}(U,m)}=\int_U \left( |g_\varepsilon(x)|^2 + \|\nabla g_\varepsilon(x)\|^2\right) m(dx)
\le \left(1+ \frac{9C^2}{\varepsilon^2}\right) m(U_\varepsilon).
\]
Thus, it remains to show that $\lim_{\varepsilon\downarrow0}\varepsilon^{-2}m(U_\varepsilon)= 0$ holds. A change of variables yields
\[
m(U_\varepsilon) = \int_{U_\varepsilon} |\det x|^{\delta-1}dx = \int_{A\times B_\varepsilon} | \det \Phi(y,v)|^{\delta-1}J(y,v)dy \otimes dv,
\]
where $J=\det \nabla \Phi$ is the Jacobian determinant. Since $\Phi$ has bounded derivative, there is a constant $\kappa$ such that $\Phi$ is $\kappa$-Lipschitz and $J\le \kappa$ holds. Together with Lemma~\ref{L:detlr} (and the fact that $\Phi(y,0)\in M$), we get
\begin{align}
m(U_\varepsilon)
&\le \kappa \int_{A\times B_\varepsilon} \left| \det \left(\Phi(y,0) + \Phi(y,v)-\Phi(y,0)\right)\right|^{\delta-1}dy \otimes dv \nonumber\\
&\le \kappa \int_{A\times B_\varepsilon} c_k\circ\Phi(y,0) \left\| \Phi(y,v)-\Phi(y,0)\right\|^{(d-k)(\delta-1)}dy \otimes dv \nonumber\\
&\le \kappa^2 \int_{A\times B_\varepsilon} c_k\circ\Phi^{-1}(y,0)\| v \|^{(d-k)(\delta-1)}dy \otimes dv \nonumber\\
&= \kappa^2 \int_A c_k\circ\Phi^{-1}(y,0)dy \int_{B_\varepsilon}\| v \|^{(d-k)(\delta-1)}dv, \label{eq:Tcap0145}
\end{align}
where $c_k(\cdot)$ is as in Lemma~\ref{L:detlr}. The integral over $A$ is finite due to the boundedness of $\Phi^{-1}$ on $A\times\{0\}$ and the Lipschitz continuity of $c_k$ on $U$, so we get
\[
m(U_\varepsilon) \le C \int_{B_\varepsilon} \|v\|^{(d-k)(\delta-1)}dv = \varepsilon^{-(d-k)(1-\delta)+n_2} C \int_{B_1} \|v\|^{-(d-k)(1-\delta)} dv
\]
for some constant $C>0$ that does not depend on $\varepsilon$. Since the integral is over $n_2$-dimensional space, the right side is finite provided
\begin{equation} \label{eq:cap01}
(d-k)(1-\delta)<n_2-2
\end{equation}
holds. But $n_2=(d-k)^2$, so \eqref{eq:cap01} is equivalent to $(d-k)(d-k-1+\delta)>2$, which holds for all $k\le d-2$ since $\delta>0$. We conclude that there is a constant $C'>0$, independent of $\varepsilon$, such that
\[
\frac{1}{\varepsilon^2} m(U_\varepsilon) \le C' \varepsilon^{-(d-k)(1-\delta)+n_2-2}.
\]
Since, as we just saw, \eqref{eq:cap01} holds, this quantity tends to zero as $\varepsilon$ tends to zero. This finishes the proof of the case $\delta>0$, $k\in\{0,\ldots,d-2\}$.

If $\delta>2$, then \eqref{eq:cap01} holds also for $k=d-1$, which takes care of this case as well. The only case that remains to consider is $\delta=2$, $k=d-1$. This is done by a slight modification of the above argument. First, $g_\varepsilon$ is now given by
\[
g_\varepsilon = \phi_\varepsilon( |\pi\circ\Phi^{-1}| ),
\]
where $\phi_\varepsilon$ is the function from Lemma~\ref{L:d2}. (Note that $n_2=(d-k)^2=1$, so that $\pi\circ\Phi^{-1}(x)$ is a real number; hence the absolute value bars.) Since $\phi_\varepsilon$ is Lipschitz it is almost everywhere differentiable by Rademacher's theorem. Hence $\nabla g_\varepsilon$ is well-defined up to a nullset. Next, instead of~\eqref{eq:Tcap01100} we need a more precise estimate. Specifically, we have the inequality
\[
\|\nabla g_\varepsilon \|
\le C\left| \phi_\varepsilon'( |\pi\circ\Phi^{-1}| ) \right|,
\]
where as before $C=\sup_{x\in U}\|\nabla(\Phi^{-1})(x)\|_{\rm op}$ is finite. In particular this gives $g_\varepsilon\in W^{1,2}(U,m)$. By the same calculations as those leading up to~\eqref{eq:Tcap0145} we then obtain, using Lemma~\ref{L:detlr},
\begin{align*}
\int_U \|\nabla g_\varepsilon(x)\|^2m(dx)
&\le C^2 \int_U \left| \phi_\varepsilon'( |\pi\circ\Phi^{-1}(x)| ) \right|^2 |\det x|\, dx \\
&= C^2 \int_{A\times B_1} \left| \phi_\varepsilon'( |v| ) \right|^2 |\det \Phi(y,v)| J(y,v)\, dy\otimes dv \\
&\le \kappa^2 C^2 \int_A c_{d-1}\circ\Phi(y,0)dy\ \int_{-\varepsilon}^\varepsilon \left| \phi_\varepsilon'( |v| ) \right|^2 |v| dv.
\end{align*}
By the property~\eqref{eq:Ld2} of $\phi_\varepsilon$ given in Lemma~\ref{L:d2}, the right side tends to zero as $\varepsilon\downarrow 0$. This concludes the proof.
\end{proof}

\section{The Muckenhoupt $A_p$ property} \label{S:A1}

Weight functions satisfying the so-called Muckenhoupt $A_p$ condition play an important role in potential theory, where they arise as precisely those weight functions for which the Hardy-Littlewood maximal operator is bounded on the corresponding weighted $L^p$~space, $1<p<\infty$, see \cite{Muckenhoupt:1972}. This and related results have far-reaching consequences, some of which are discussed in~\cite{Torchinsky:1986fk,Turesson:2000,Kilpelainen:1994}. In this section we prove that the weight function $w(x)=|\det x|^\alpha$ lies in the Muckenhoupt~$A_p$ class for certain combinations of $p$ and $\alpha$. Our result generalizes the case $d=1$, for which the result is known, in a striking way. We let $|A|=\int_A dx$ denote the Lebesgue measure of a measurable subset $A\subset \M^d$.

\begin{theorem}[Muckenhoupt property] \label{T:Ap}
Let $\alpha\in\R$ and define $w(x) = |\det x|^\alpha$.
\begin{enumerate}
\item If $-1<\alpha\le 0$, then $w$ lies in the \emph{Muckenhoupt $A_1$ class}. That is, there is a constant $C>0$ depending only on $d$ and $\alpha$, such that
\begin{equation} \label{eq:A1}
\frac{1}{|B|} \int_B w(x) dx  \le C \inf_{x\in B} w(x)
\end{equation}
for every ball $B\subset\M^d$.

\item If $-1<\alpha<p-1$, $p>1$, then $w$ lies in the \emph{Muckenhoupt $A_p$ class}. That is, there is a constant $C>0$ depending only on $d$, $\alpha$ and $p$, such that
\[
\left( \frac{1}{|B|} \int_B w(x) dx \right) \left( \frac{1}{|B|} \int_B w(x)^{-1/(p-1)} dx\right)^{p-1} \le C
\]
for every ball $B\subset\M^d$.
\end{enumerate}
\end{theorem}

Once part~(i) has been proved, part~(ii) follows directly from \cite[Proposition~IX.4.3]{Torchinsky:1986fk}. It thus suffices to prove part~(i), which will occupy the rest of this section. We first introduce some notation. Let ${\bf D}^d_+$ denote the set of diagonal matrices with nonnegative and ordered diagonal elements,
\[
{\bf D}^d_+ = \{\Diag(\sigma) : \sigma \in \R^d,\ \sigma_1\ge\sigma_2\ge\dots\ge\sigma_d\ge 0\}.
\]
The open ball centered at $x\in\M^d$ with radius $r>0$ is denoted by $B(x,r)$. Its intersection with the nonsingular matrices is denoted by $B_*(x,r)$. That is,
\[
B(x,r)=\{y\in\M^d:\|x-y\|<r\}, \qquad B_*(x,r) = \{y\in B(x,r): \det y \ne 0\}.
\]
The proof of the Muckenhoupt property is somewhat involved (but nonetheless mostly elementary), due to the relatively complicated geometric structure of the set $\partial E$, which is where the weight function becomes singular. The main idea is to change variables using the $QR$-decomposition and integrate over the product space $O(d)\times T(d)$ instead of $\M^d$. Unfortunately, balls in $\M^d$ do not always map to balls (or comparable shapes) in $O(d)\times T(d)$, and this is where the main complications arise. The resolution to this issue resides in Lemma~\ref{L:A12} below, which relies on a detailed analysis of the mapping taking $x$ to its $QR$-decomposition.

We start with a lemma that establishes an inequality similar to~\eqref{eq:A1}, where the balls $B$ are replaced by sets of the form $U\cdot K=\{QR:Q\in U, R\in K\}$, with $U\subset O(d)$ measurable and $K\subset T(d)$ a cube.

\begin{lemma} \label{L:A11}
Let $-1<\alpha\le 0$. Then there is a constant $C_1>0$, depending only on $d$ and $\alpha$, such that the inequality
\[
\int_{U\cdot K} w(x) dx \le C_1|U\cdot K| \inf_{x\in U\cdot K} w(x)
\]
holds for any measurable subset $U\subset O(d)$ and any cube $K\subset T(d)$.
\end{lemma}

\begin{proof}
Pick a cube $K=\{R\in T(d): R_{ij}\in I_{ij}, \ i\le j\}$, where the $I_{ij}$ are bounded intervals, and let $U\subset O(d)$ be measurable. By Lemma~\ref{L:c-o-v} we have
\[
\int_{U\cdot K} w(x)dx = \mu(U) \int_K \prod_{i=1}^d R_{ii}^{d-i+\alpha}dR = \mu(U) \prod_{i<j}|I_{ij}| \prod_{i=1}^d \int_{I_{ii}} t^{d-i+\alpha}dt,
\]
and similarly $|U\cdot K|=\mu(U)\prod_{i<j}|I_{ij}| \prod_{i=1}^d\int_{I_{ii}} t^{d-i}dt$. We also have
\[
\inf_{x\in U\cdot K}w(x) = \inf_{R\in K} \prod_{i=1}^dR_{ii}^\alpha=\prod_{i=1}^d\inf_{t\in I_{ii}}t^\alpha.
\]
Hence the result follows from the following Claim:

{\it Let $\alpha\in(-1,0]$ and $\beta\ge 0$. Then there is a constant $C_{\alpha,\beta}$ such that for every bounded interval $I\subset(0,\infty)$, we have}
\[
\int_I t^{\alpha+\beta}dt \le C_{\alpha,\beta} \int_I t^\beta dt  \  \inf_{t\in I} t^\alpha.
\]

To prove the Claim it suffices to consider $I=(a,b)$ with $0\le a<b$. We obtain:
\begin{align*}
\int_I t^{\alpha+\beta}dt &= \frac{1}{\alpha+\beta+1}\left(b^{\beta+1}-\left(\frac{a}{b}\right)^\alpha a^{\beta+1}\right) b^\alpha \\
&\le \frac{1}{\alpha+\beta+1} \left( b^{\beta+1}-a^{\beta+1}\right) b^\alpha \\
&= \frac{\beta+1}{\alpha+\beta+1} \int_I t^\beta dt \ \inf_{t\in I} t^\alpha,
\end{align*}
as required.
\end{proof}

Consider now balls $B(\Sigma,r)$, where the diagonal elements of $\Sigma\in{\bf D}^d_+$ are either ``large'' (comparable to the radius $r$) or zero. The following result reduces the proof that~\eqref{eq:A1} holds for balls of this form to an application of Lemma~\ref{L:A11}. In the statement of condition~\eqref{eq:Sig} below, we use the convention that $\sigma_0=\infty$ and that $i$ runs over $\{0,\ldots,d\}$.

\begin{lemma} \label{L:A12}
Suppose $\Sigma\in {\bf D}^d_+$ and $r>0$ satisfy the following property, where $\sigma\in\R^d$ is the vector of diagonal elements of $\Sigma$:
\begin{equation}\label{eq:Sig}
\begin{array}{l}
\text{There is an index } n\in\{0,1,\ldots,d\} \text{ such that }   \\
\sigma_i> 18d r \text{ for all } i\le n \text{, and } \sigma_i=0 \text{ for all } i> n.
\end{array}
\end{equation}
Then there is a measurable subset $U\subset O(d)$ and a cube $K\subset T(d)$ such that
\[
B_*(\Sigma,r) \subset U\cdot K \subset B_*(\Sigma, C_2r),
\]
where $C_2$ is a positive constant that only depends on $d$.
\end{lemma}

\begin{proof}
The problem of finding the advertised constant $C_2$ can be reduced to proving the following Claim, where $e_1,\ldots,e_d$ denote the canonical unit vectors in $\R^d$:

{\it There is a constant $C_3$, depending only on $d$, such that the following holds: For any $x\in B_*(\Sigma,r)$, let $x=QR$ be its $QR$-decomposition, and let $q_1,\ldots,q_d$ be the columns of $Q$. Then the inequalities $\|R - \Sigma\|< C_3 r$ and $|q_i-e_i| < r\sigma_i^{-1}C_3$ hold for all $i\in\{1,\ldots,n\}$, where $n$ is the index from condition~\eqref{eq:Sig}.}

Let us show how the statement of the lemma follows from this claim. Define $K$ to be the cube in $T(d)$ centered at $\Sigma$ with side $2C_3r$, i.e.
\[
K=\{R\in T(d) : \max_{i,j}|R_{ij}-\Sigma_{ij}| < C_3r\},
\]
and let $U\subset O(d)$ be given by
\[
U = \left\{ Q=(q_1,\ldots,q_d) \in O(d) :  |q_i - e_i| < \frac{r}{\sigma_i}C_3, i=1,\ldots,n\right\}.
\]
The Claim then directly implies $B_*(\Sigma,r)\subset U\cdot K$. We thus need to show that it also implies $U\cdot K\subset B_*(\Sigma,C_2r)$ for some constant $C_2>0$ that only depends on $d$. To this end, observe that for any $x=QR\in U\cdot K$ we have, by the triangle inequality, the rotation invariance of $\|\cdot\|$, and the definition of~$K$,
\begin{align*}
\| x - \Sigma \|
&\le \|Q(R-\Sigma)\| + \| (Q-\Id) \Sigma\| \\
&= \|R-\Sigma\| + \| (Q-\Id) \Sigma\| \\
&< \sqrt{d(d+1)/2} C_3 r + \| (Q-\Id) \Sigma\|.
\end{align*}
Furthermore, since $\sigma_i=0$ for $i>n$, we have $\| (Q-\Id) \Sigma\|^2=\sigma_1^2|q_1-e_1|^2+\dots\sigma_n^2|q_n-e_n|^2$. We then deduce from the Claim that $\|(Q-\Id)\Sigma\|<\sqrt{n} C_3 r$, and consequently
\[
\| x - \Sigma \| < C_2 r, \qquad\text{where}\qquad C_2 = \left( \sqrt{d(d+1)/2}  + \sqrt{d} \right) C_3.
\]

We are thus left with proving the Claim. Since it is vacuously true for $n=0$, we can assume $n\ge 1$. The proof relies on a rather careful analysis of the Gram-Schmidt orthogonalization procedure for obtaining the $QR$-decomposition of a generic matrix $x\in B_*(\Sigma,r)$, so we briefly recall this procedure. To improve readability, we temporarily (for this proof only) adopt the notation $\langle y,z\rangle = y^\top z$ for $y,z\in\R^d$. Fix $x\in B_*(\Sigma,r)$ and let $x_1,\ldots,x_d$ be the columns of $x$. To obtain the $QR$-decomposition of $x$, one defines
\[
u_1 = x_1, \qquad q_1 = \frac{u_1}{|u_1|},
\]
and, if $q_1,\ldots,q_{j-1}$ have been defined,
\begin{equation} \label{eq:GS2}
u_j = x_j - \sum_{i=1}^{j-1} \langle q_i,x_j\rangle q_i, \qquad q_j = \frac{u_j}{|u_j|}.
\end{equation}
The vectors $q_1,\ldots,q_d$ obtained in this way are the columns of $Q$, and $R$ is given by $R_{ij}=\langle q_i,x_j\rangle$, $i\le j$.

We now proceed with the proof of the Claim. Recall that $e_1,\ldots,e_d$ are the canonical unit vectors in $\R^d$. Since $x\in B_*(\Sigma,r)$, we have $x_i = \sigma_ie_i + h_i$, where $h_i$ is a vector in $\R^d$ with $|h_i|< r$. Also let $a=5+18d$ denote the constant appearing in condition~\eqref{eq:Sig}.

Fix $j\in\{1,\ldots,n\}$, and suppose we have proved the following:
\begin{equation}\label{eq:inda}
\text{For all }i\le j-1 \text{ and all }k>i, \quad |R_{ik}| < 3r.
\end{equation}
Then \eqref{eq:GS2} and the inequalities $|x_j|\ge \sigma_j-|h_j|\ge\sigma_j-r$ imply
\begin{equation} \label{eq:uj}
|u_j| \ge \sigma_j - r - \sum_{i=1}^{j-1} |R_{ij}| \ge \sigma_j - r (1+3(j-1)).
\end{equation}
Moreover, for $k>j$ we use \eqref{eq:GS2}, \eqref{eq:inda}, and the fact that $|\langle x_j, x_k\rangle| < \sigma_j r+ \sigma_k r + r^2$ to get
\[
|\langle u_j, x_k\rangle| \le |\langle x_j, x_k\rangle| + \sum_{i=1}^{j-1} |R_{ij}| \, |R_{ik}|
< \sigma_j r+ \sigma_k r + r^2(1 + 9(j-1)).
\]
Together with \eqref{eq:uj} this yields
\begin{align*}
|R_{jk}| =\frac{|\langle u_j, x_k\rangle|}{|u_j|}
&\le r \frac{\sigma_j + \sigma_k + r(1 + 9(j-1)) }{ \sigma_j - r(1 + 3(j-1)) } \\
&=  r \frac{1 + \sigma_k/\sigma_j + (r/\sigma_j)(1 + 9(j-1)) }{ 1 - (r/\sigma_j)(1 + 3(j-1)) } \\
&< r \frac{2 + a^{-1}(1 + 9(j-1)) }{ 1 - a^{-1}(1 + 3(j-1)) },
\end{align*}
where in the last step we used that $\sigma_k\le\sigma_j$ (since $k>j$) and $\sigma_j>ar$ (since $j\le n$). Since $a=18d$, the right side is at most $3r$, as one readily verifies. We deduce that~\eqref{eq:inda} holds with $j$ replaced by $j+1$, and since it is vacuously true for $j=1$ it follows by induction that it holds for all $j\in\{1,\ldots,n+1\}$.

We now use this result to bound $|R_{jj}-\sigma_j|$ for $j\in\{1,\ldots,n\}$. To this end, write
\[
|R_{jj} - \sigma_j| = \frac{1}{|u_j|} \left| |x_j|^2 - \sum_{i=1}^{j-1} \langle q_i,x_j\rangle^2 - \sigma_j |u_j| \right|
\le
\frac{ \left| |x_j|^2 - \sigma_j|u_j| \right| + 9(j-1)r^2}{|u_j|},
\]
using that $\langle q_i,x_j\rangle^2=|R_{ij}|^2<9r^2$ due to~\eqref{eq:inda}. Moreover, we have
\[
\left| |x_j|^2 - \sigma_j|u_j| \right| = \left| \sigma_j^2 - \sigma_j|u_j| + |h_j|^2 + 2\sigma_j\langle e_j,h_j\rangle \right|
\le  \sigma_j \left| \sigma_j - |u_j| \right| + r^2 + 2\sigma_j r,
\]
and by the reverse triangle inequality,
\begin{equation} \label{eq:revtr}
\left| \sigma_j - |u_j| \right| \le |\sigma_je_j-u_j| \le r + 3(j-1)r.
\end{equation}
Assembling the pieces and using the bound~\eqref{eq:uj} gives
\[
|R_{jj} - \sigma_j| \le \frac{\sigma_j r (1 + 3(j-1)) + r^2 + 2\sigma_j r + 9(j-1)r^2}{ \sigma_j - r(1+3(j-1)) }.
\]
Dividing the numerator and denominator by $\sigma_j$ and using that $\sigma_j>ar$, we finally arrive at
\begin{align*}
|R_{jj} - \sigma_j| &\le r \frac{ 1 + 3(j-1) + a^{-1} + 2  + a^{-1}9(j-1)}{ 1 - a^{-1}(1+3(j-1)) } \\
&= r \frac{ 3 + 3(j-1) + a^{-1}(1  +9(j-1))}{ 1 - a^{-1}(1+3(j-1)) } \\
&< r \frac{ 3 + 18d}{ 5 }.
\end{align*}
The only elements of $R$ that remain to analyze are $R_{ij}$ for $i\le j$ and $j>n$. But $x_j=h_j$ for these $j$, so $|R_{ij}|=|\langle q_i,x_j\rangle| \le |h_j|<r$. We are now able to estimate $\|R-\Sigma\|$ as follows:
\begin{align*}
\| R-\Sigma \|
& \le \sum_j |R_{jj}-\sigma_j| + \sum_{i< j} |R_{ij}| \\
& \le r \left( \frac{ 3 + 18d}{ 5 }\times d + 3\times \frac{n(n-1)}{2}+ d(d-n) \right).
\end{align*}
A bound solely in terms of $d$ is then easily obtained. For instance, we may take
\begin{equation}\label{eq:bdRS}
\| R-\Sigma \| \le r d \left( \frac{ 3 + 18d}{ 5 } + 3 \times\frac{(d-1)}{2}+ d \right).
\end{equation}

Let us now focus on bounding $|q_j-e_j|$, $j\in\{1,\ldots,n\}$. The calculations are similar to the ones used to bound $|R_{jj}-\sigma_j|$ above, but slightly simpler. We have
\begin{align*}
|q_j - e_j| &= \frac{1}{|u_j|} \left|  x_j - \sum_{i=1}^{j-1} \langle q_i,x_j\rangle q_i - |u_j|e_j \right| \\
&\le \frac{1}{|u_j|}\Big( \big| \sigma_j - |u_j| \big| + r + 3(j-1)r \Big) \\
&\le \frac{2}{|u_j|}\Big(r + 3(j-1)r \Big),
\end{align*}
using~\eqref{eq:revtr} in the last step. Using again~\eqref{eq:uj} together with $\sigma_j> ar$,
\[
|q_j - e_j| \le \frac{r}{\sigma_j} \times\frac{2+6(j-1)}{1-a^{-1}(1+3(j-1))} < \frac{r}{\sigma_j} \times11d.
\]
The Claim, and hence the lemma, is now proved, if for $C_3$ we take the maximum of $11d$ and the constant in~\eqref{eq:bdRS}.
\end{proof}

Next, Lemma~\ref{L:A13} below implies that the proof of \eqref{eq:A1} for any ball whose center lies in ${\bf D}^d_+$ reduces to an application of Lemma~\ref{L:A12}. It uses the following simple observation.

\begin{lemma} \label{L:A131}
Let $a>0$ and $k\in\{0,1,2,\ldots\}$. We have
\[
1 + a + a(1+a) + \cdots + a(1+a)^k = (1+a)^{k+1}.
\]
\end{lemma}

\begin{proof}
The result clearly holds for $k=0$. If it holds for $k-1$, we get
\[
1 + a + a(1+a) + \cdots + a(1+a)^k=(1 + a)(1 + a + \cdots + a(1+a)^{k-1})=(1+a)^{k+1},
\]
showing that it holds for $k$ as well.
\end{proof}

\begin{lemma} \label{L:A13}
Pick any $\Sigma=\Diag(\sigma) \in {\bf D}^d_+$ and $r>0$. There is a matrix $\Sigma'=\Diag(\sigma') \in {\bf D}^d_+$ and a real number $r'>0$ that satisfy the condition~\eqref{eq:Sig} (with $\sigma$ replaced by $\sigma'$, and $r$ by $r'$), such that
\[
B(\Sigma,r) \subset B(\Sigma',r') \qquad\text{and}\qquad r' \le (1+18d)^d r.
\]
\end{lemma}

\begin{proof}
Let $a>0$ be a constant to be determined later. Suppose for some index $i\in\{1,\ldots,d\}$, we have $\sigma_i > a(1+a)^{d-i} r$. Let $n$ be the largest such index, and define
\[
\sigma' = (\sigma_1,\ldots,\sigma_n,0,\ldots,0), \qquad r' = r(1+a)^{d-n}.
\]
Then, since $\sigma_i \le  a(1+a)^{d-i}r$ for all $i>n$,
\begin{align*}
|\sigma - \sigma'| &\le \sigma_{n+1}+\dots+\sigma_d \\
&\le r\left( a(1+a)^{d-n-1}+\dots+a(1+a)+a \right) \\
&= r(1+a)^{d-n}-r \\
&= r' -r.
\end{align*}
The triangle inequality yields $B(\Sigma,r)\subset B(\Sigma',r')$, where $\Sigma'=\Diag(\sigma')$. Setting $a=18d$, we see that $\Sigma'$, $r'$ satisfy condition~\eqref{eq:Sig}.

It remains to consider the case where $\sigma_i \le a(1+a)^{d-i} r$ for all $i\in\{1,\ldots,d\}$. In this case any $x\in B(\Sigma,r)$ satisfies
\[
\|x\| \le r + \|\Sigma\| \le r + r\left( a(1+a)^{d-1}+\dots+a(1+a)+a \right) = r(1+a)^d,
\]
so that $B(\Sigma,r)\subset B(0,r(1+a)^d)$. With $r'= r(1+a)^d$ and $n=0$, condition~\eqref{eq:Sig} is again satisfied for $a=18d$. This finishes the proof.
\end{proof}

\begin{proof}[Proof of Theorem~\ref{T:Ap}{\rm (i)}]
The proof of \eqref{eq:A1} is now straightforward. Indeed, pick any ball $B=B(x,r)$, and let $x=U\Sigma V^\top$ be a singular value decomposition of $x$. Then $B(x,r)=U \cdot B(\Sigma,r)\cdot V^\top$, and together with the invariance of Lebesgue measure under orthogonal transformations and the fact that $\det y = \det(U^\top yV)$ for any $y\in\M^d$, this leads to the equalities
\begin{align*}
\int_{B(x,r)} w(x)dx &= \int_{B(\Sigma,r)} w(x)dx, \\[3mm]
|B(x,r)|&=|B(\Sigma,r)|, \\[4mm]
\inf_{x\in B(x,r)}w(x) &= \inf_{x\in B(\Sigma,r)}w(x).
\end{align*}
Consequently (and using that $\partial E$ is a nullset), it suffices to prove~\eqref{eq:A1} for $B$ replaced by $B_*=B_*(\Sigma,r)$ with $\Sigma\in{\bf D}^d_+$. We then have the following chain of inequalities, where we set $B'_*=B_*(\Sigma',r')$ with $\Sigma'$ and $r'$ from Lemma~\ref{L:A13}, and where $U$, $K$, and $C_2$ are obtained by applying Lemma~\ref{L:A12} to $\Sigma'$, $r'$.
\begin{align*}
\int_{B_*} w(x)dx &\le \int_{U\cdot K} w(x)dx 	& (B_* \subset B_*' \subset U\cdot K)\\
&\le C_1 |U\cdot K| \inf_{x\in U\cdot K} w(x)	& \text{(Lemma~\ref{L:A11})}\\
&\le C_1 |U\cdot K| \inf_{x\in B_*} w(x)		& (B_* \subset U\cdot K)\\
&\le C_1 |B_*(\Sigma', C_2 r')| \inf_{x\in B_*} w(x)	& (U\cdot K \subset B_*(\Sigma',C_2 r'))\\
&= C_1 C_2^{d^2}(1+18d)^{d^3} |B_*| \inf_{x\in B_*} w(x). & (r' \le (1+18d)^dr)
\end{align*}
This proves that~\eqref{eq:A1} holds with $C=C_1C_2^{d^2}(1+18d)^{d^3}$.
\end{proof}

\appendix

\section{Proof of the integration by parts formula} \label{A:ibp}

In this section we give a proof of the integration by parts formula, Theorem~\ref{T:ibp}, which we now restate for the reader's convenience:

{\it
Suppose $\delta>0$, and consider $f\in C^1_c(E)$ and $G\in C^1(E; \M^d)$. If $\delta\le 1$, assume that $G(x)$ is tangent to $\partial E$ at $x$ for all $x\in \partial E$. If $\delta<1$, assume in addition that $G(x)\bullet x^{-\top}$ is locally bounded. Then
\[
\langle \nabla f, G\rangle  =  \langle f, \nabla^*G\rangle.
\]
}

Throughout the proof, let $K$ be the (compact) support of $f$. For $\varepsilon\ge 0$, define
\[
U_\varepsilon = \{x\in \M^d: \det x >\varepsilon\} \qquad\text{and}\qquad \nu(x) = -\frac{\nabla \det(x)}{\|\nabla\det(x)\|}, \quad x\in U_0.
\]
For $\varepsilon>0$, $U_\varepsilon$ has smooth boundary with outward unit normal $\nu(x)$ at $x\in\partial U_\varepsilon$. For any smooth function $h:U_0\to\R$ such that the integrals are well-defined, the standard integration by parts formula yields, for each~$\varepsilon>0$,
\begin{align}\label{eq:ibppr00}
\int_{U_\varepsilon} \nabla f(x) \bullet G(x) h(x) dx = \int_{\partial U_\varepsilon} f(x)&h(x)G(x)\bullet \nu(x)d\sigma_\varepsilon(x) \nonumber \\ 
& - \int_{U_\varepsilon} f(x)\nabla\bullet (Gh)(x) dx,
\end{align}
where $\sigma_\varepsilon$ denotes the surface area measure on $\partial U_\varepsilon$.

{\it Case 1: $\delta>1$.}
Take $h(x)=\det(x)^{\delta-1}$. As $\varepsilon\downarrow0$, the left side of~\eqref{eq:ibppr00} tends to $\int_E\nabla f(x)\bullet G(x)m(dx)$ by dominated convergence. Let $C>0$ be such that $|f(x)|\|G(x)\|\le C$ for all $x\in K$. The absolute value of the boundary term is then dominated by
\[
C\varepsilon^{\delta-1} \int_{\partial U_\varepsilon \cap K}d\sigma_\varepsilon(x),
\]
using also that $h(x)=\varepsilon^{\delta-1}$ for $x\in\partial U_\varepsilon$. It is easy to see that $\sigma_\varepsilon(K)$ remains bounded as $\varepsilon\downarrow 0$, so we conclude that the boundary term vanishes in the limit. Consider now the second term on the right side of~\eqref{eq:ibppr00}. The product rule yields
\[
\nabla \bullet (Gh) = h \nabla\bullet G + G\bullet\nabla h.
\]
By dominated convergence, $\int_{U_\varepsilon}f(x)\nabla\bullet G(x)h(x)dx\to \int_Ef(x)\nabla\bullet G(x)m(dx)$. Moreover, we have $G(x)\bullet\nabla h(x)=(\delta-1)G(x)\bullet\nabla \det(x) \det(x)^{\delta-2}$. Since $\det(x)^{\delta-2}dx$ is a Radon measure due to Theorem~\ref{T:Radon} and the fact that $\delta>1$, we may again use dominated convergence to get
\[
\int_{U_\varepsilon} f(x)G(x)\bullet \nabla h(x)dx \to (\delta-1)\int_E f(x)G(x)\bullet x^{-\top} m(dx).
\]
(Here we used the equality $\nabla \det(x) \det(x)^{\delta-2}dx=x^{-\top}m(dx)$.) Assembling the pieces gives the desired formula~\eqref{eq:ibp}.

{\it Case 2: $\delta=1$.}
We again take $h(x)=\det(x)^{\delta-1}\equiv 1$. Except for the boundary term, everything works as in the case $\delta>1$, if we just note that $\nabla h=0$. Letting $C$ be a bound on $|f(x)|$ over $K$, the boundary term is bounded above by
\[
C \sigma_\varepsilon(K) \sup_{x\in U_\varepsilon\cap K} G(x)\bullet\nu(x).
\]
Using that $G(x)$ is tangent to $\partial E$ at every $x\in\partial E$ it is not hard to show that the supremum tends to zero. Hence~\eqref{eq:ibp} is established.

{\it Case 3: $\delta<1$.}
Things are now a bit more complicated due to the fact that $\det(x)^{\delta-1}$ blows up at $\partial E$. To get around this, for each $n$ let $\tau_n$ be a smooth, nondecreasing function satisfying the following properties:
\[
\tau_n(t)\le t\wedge n, \quad \tau_n(t)=t \text{ for } t\le n-1, \quad \tau_n(t)=n \text{ for } t\ge n+1
\]
\[
0\le \tau'_n \le 1, \qquad \tau_n(t) \uparrow t \text{ and } \tau_n'(t) \uparrow 1 \text{ as } n\to\infty.
\]
In~\eqref{eq:ibppr00} we now take $h=h_n$, where $h_n=\tau_n \circ w$ and $w(x)=\det(x)^{\delta-1}$. We first hold $n$ fixed and let $\varepsilon\downarrow 0$. The left side of \eqref{eq:ibppr00} converges to $\int_E \nabla f(x) \bullet G(x) h_n(x) dx$ by dominated convergence. The boundary term on the right side will vanish by the same argument as in the case $\delta=1$. The integrand in the second term on the right side is in fact bounded, since by the properties of $\tau_n$,
\[
\nabla h_n(x)=(\delta-1) (\tau_n'\circ w(x))\nabla\det(x) \det(x)^{\delta-2} \le (\delta-1)\nabla\det(x) (n+1)^{\frac{\delta-2}{\delta-1}}.
\]
Dominated convergence gives the limit $\int_Ef(x)\nabla\bullet(Gh_n)(x)dx$. Combining these results and applying the product rule gives the formula
\begin{equation} \label{eq:AppAeq1}
\int_E \nabla f(x) \bullet G(x) h_n(x) dx = -\int_E f(x)\nabla\bullet G(x)h_n(x) -\int_E f(x) G(x)\bullet \nabla h_n(x)dx.
\end{equation}
The final step is to send $n$ to infinity. The left side of~\eqref{eq:AppAeq1} converges to $\int_E \nabla f(x) \bullet G(x) m(dx)$ by dominated convergence, since $h_n=\tau_n\circ w \uparrow w$. For the first term on the right side of~\eqref{eq:AppAeq1}, we similarly have $\int_E f(x)\nabla\bullet G(x)h_n(x)(dx)\to \int_E f(x)\nabla\bullet G(x)m(dx)$. Finally, for the second term on the right side of~\eqref{eq:AppAeq1}, note that
\[
f(x)\,G(x)\bullet \nabla h_n(x) = (\delta-1) f(x) \tau_n'\circ w(x)\, G(x)\bullet x^{-\top}\, w(x).
\]
This is bounded in absolute value by a constant times $|G(x)\bullet x^{-\top}|\nv 1_K\, w(x)$, which is integrable since $|G(x)\bullet x^{-\top}|$ is locally bounded by hypothesis. Thus dominated convergence yields $\int_E f(x) \,G(x)\bullet \nabla h_n(x) dx \to (\delta-1)\int_E f(x) G(x)\bullet x^{-\top} w(x)dx$, and hence the result.

\section{The space $W^{1,2}(\Omega,m)$} \label{A:W12}

In this appendix we review some basic properties of the weighted Sobolev space $W^{1,2}(\Omega,m)$ introduced in Section~\ref{S:MU}, as well some related results. The material is not new---we collect the results here for ease of reference.

Besides $W^{1,2}(\Omega,m)$ there is occasionally a need to consider spaces $W^{1,2}(U,m)$ for open sets $U\subset \M^d$ different from $\Omega$. Here the following subtlety arises: If $U\cap\partial E\ne\emptyset$ and $\delta>1$, then $f\in L^2(\Omega,m)$ need not lie in $L^1_{\rm loc}(U)$, see \cite[Example~1.7]{Kufner/Opic:1984}. Thus we cannot speak about its distributional gradient. In this case we therefore define
\[
W^{1,2}(U ,m) = \{ f\in L^2(\Omega ,m)\cap L^1_{\rm loc}(U) :  Df \in L^2(\Omega ,m; \M^d)\cap L^1_{\rm loc}(U)\}.
\]
It is clear that for two open subsets $U$, $V$ satisfying $U\subset V$ and an element $f\in W^{1,2}(V,m)$, we have $f|_U\in W^{1,2}(U,m)$. To alleviate notation we simply write $f\in W^{1,2}(U,m)$. If the open set $U\subset \M^d$ has compact closure in $\Omega$, we have $C^{-1}\le (\det x)^{\delta-1}\le C$ for some constant $C>1$ and all $x\in U$. Hence
\begin{equation} \label{eq:normequiv}
\| \cdot \|_{W^{1,2}(U,m)} \text{ and } \| \cdot \|_{W^{1,2}(U,dx)} \text{ are equivalent},
\end{equation}
and the unweighted space $W^{1,2}(U,dx)$ coincides with $W^{1,2}(U,m)$. This has several useful consequences.

\begin{lemma} \label{L:W12props}
Let $f\in W^{1,2}(\Omega,m)$. The following statements hold.
\begin{enumerate}
\item Mollification: Let $\psi$ be a mollifier and set $\psi_\varepsilon(x)=\varepsilon^{-d}\psi(x/\varepsilon)$. If $f\in W^{1,2}(\Omega,m)$ has compact support in $\Omega$, then $\lim_{\varepsilon\to0}\psi_\varepsilon*f= f$ in $W^{1,2}(\Omega,m)$.
\item Stability under truncation: For $f\in W^{1,2}(\Omega,m)$ we have $\lim_{n\to\infty}f\wedge n= f$ in $W^{1,2}(\Omega,m)$.
\item Let $f,g\in W^{1,2}(\Omega,m)$ with $f$ and $g$ bounded. Then $fg\in W^{1,2}(\Omega,m)$ and we have
\begin{equation} \label{eq:Dfg}
D(fg) = fDg + gDf.
\end{equation}
\item The set $\{f\in W^{1,2}(\Omega,m): f \text{ is bounded with bounded support}\}$ is dense in $W^{1,2}(\Omega,m)$.
\end{enumerate}
\end{lemma}

\begin{proof}
(i): Let $K$ be the support of $f$ and pick an open set $U\Subset\Omega$ with $K\subset U$. Then $\psi_\varepsilon*f\to f$ in $W^{1,2}(U,dx)$ by standard results in the unweighted case, see \cite[Lemma~3.16]{Adams/Fournier:2003}. Thus, for all sufficiently small $\varepsilon>0$ we have $\|f-\psi_\varepsilon*f\|_{W^{1,2}(\Omega,m)}=\|f-\psi_\varepsilon*f\|_{W^{1,2}(U,m)}\to 0$.

(ii): Pick a test function $\phi\in C^\infty_c(\Omega)$ with support $K$, and choose an open set $U\Subset\Omega$ with $K\subset U$. Then $f$ lies $W^{1,2}(U,dx)$, as does $g=f\wedge n$ by \cite[Lemma~1.7.1]{Mazia:1985fk}, with weak derivative $Dg=Df\1{f<n}$ on $U$. Thus
\[
\int_Ug\,\nabla\phi dx = -\int_UDf\1{f<n} \phi dx,
\]
and since $\phi=0$ outside $U$, this equality holds with $U$ replaced by $\Omega$. Hence $Dg=Df\1{f<n}$ on $\Omega$, and the result follows by monotone convergence.

(iii): First note that $fg$ and $fDg+gDf$ lie in $L^2(\Omega,m)$, so it remains to prove~\eqref{eq:Dfg}. To this end, pick a test function $\phi\in C^\infty_c(\Omega)$ with support $K$, and choose an open set $U\Subset\Omega$ with $K\subset U$. Let $\phi_\varepsilon$ be as in~(i). Then $f*\phi_\varepsilon\to f$ and $g*\phi_\varepsilon\to g$ in $W^{1,2}(U,dx)$. Hence $f*\phi_\varepsilon\,g*\phi_\varepsilon\,\to fg$, $f*\phi_\varepsilon\, \nabla(g*\phi_\varepsilon)\to fDg$ and $g*\phi_\varepsilon\, \nabla(f*\phi_\varepsilon)\to gDf$, all in $L^1(U)$, which yields
\begin{align*}
\int_U fg\, \nabla\phi dx
&= \lim_{\varepsilon\to0} \int_U f*\phi_\varepsilon\ g*\phi_\varepsilon\ \nabla\phi\ dx \\
&= - \lim_{\varepsilon\to0} \int_U \left( f*\phi_\varepsilon\, \nabla(g*\phi_\varepsilon) + g*\phi_\varepsilon\, \nabla(f*\phi_\varepsilon)\right)\phi\, dx \\
&= - \int_U \left( f Dg + gDf\right) \phi\, dx.
\end{align*}
This gives the result since $\phi=0$ outside $U$.

(iv): For $f\in W^{1,2}(\Omega,m)$ and $\varepsilon>0$, let $U\subset \Omega$ be such that $\int_{U^c}(|f|^2+\|Df\|^2)m(dx)\le\varepsilon$. Let $\phi\in C^\infty_c(M^d)$ be a smooth cutoff function with $\phi=1$ on $U\subset\M^d$ and $\|\nabla\phi\|\le 1$. Then $g=\phi f$ has bounded support, lies in $W^{1,2}(\Omega,m)$, and satisfies $\|f-g\|_{W^{1,2}(\Omega,m)}\le\varepsilon$. An application of~(ii) now yields the result.
\end{proof}

\begin{lemma} \label{L:W12fg}
Let $f\in W^{1,2}(\Omega,m)$ with $|f|\le 1$. Then for any $\varepsilon>0$ there is a constant $C>0$, depending only on $f$ and $\varepsilon$, such that
\[
\| f g \|_{W^{1,2}(\Omega,m)}^2 \le C\|g\|^2_{W^{1,2}(\Omega,m)} + \varepsilon
\]
holds for all $g\in W^{1,2}(\Omega,m)$ with $|g|\le 1$.
\end{lemma}

\begin{proof}
Due to Lemma~\ref{L:W12props}(iii) the product rule holds for $f$ and $g$. Therefore we have $|fg|^2 + \|D(fg)\|^2 \le |g|^2 + 2 \|Dg\|^2 + 2 |g|^2\|Df\|^2$ and hence
\[
\| f g \|_{W^{1,2}(\Omega,m)}^2 \le 2 \| g \|_{W^{1,2}(\Omega,m)}^2  + 2\int_\Omega |g|^2\,\|Df\|^2\, m(dx).
\]
Let $A=\{x\in \Omega: \| Df(x) \| \le \kappa\}$, where $\kappa$ is chosen large enough that $\int_{\Omega\setminus A}\|Df\|^2\,m(dx)\le\varepsilon/2$. Then
\[
\int_\Omega |g|^2\|Df\|^2\, m(dx) \le \kappa^2 \int_A |g|^2\, m(dx) + \int_{\Omega\setminus A} \|Df\|^2\, m(dx) \le \kappa^2 \|g\|^2_{W^{1,2}(\Omega,m)} + \varepsilon/2.
\]
The result now follows with $C=2+2\kappa^2$.
\end{proof}

\begin{lemma} \label{L:W12UVfg}
Consider open subsets $U\subset V$ and a function $\phi\in C^\infty_c(V)$ with $\phi=0$ on $V\setminus U$. Then there is a constant $C>0$ such that
\[
\| \phi g \|^2_{W^{1,2}(V,m)} \le C \| g\|^2_{W^{1,2}(U,m)}
\]
holds for all $g\in W^{1,2}(U,m)$.
\end{lemma}

\begin{proof}
Let $\kappa$ denote a bound on $\phi^2$ and $\|\nabla \phi\|^2$. We then have $|\phi g|^2+\|D(\phi g)\|^2\le 3\kappa |g|^2+2\kappa\|D g\|^2$ on $U$. Since $\phi=0$ on $V\setminus U$, the result follows with $C=3\kappa$.
\end{proof}

\section{Tube segments} \label{A:tube}

The proofs of some of our results require the notion of a {\em tube segment}, which we now introduce. For a background on the relevant notions from differential geometry the reader is referred to \cite{Lee:2003}. Let $M$ be a smooth $n_1$-dimensional embedded submanifold of $\R^n$ ($n_1<n$) and set $n_2=n-n_1$. For the applications in the present paper, $\R^n$ is identified with $\M^d$, and $M$ is one of the smooth manifolds $M_k$, $k\in\{0,\ldots,d-1\}$, consisting of rank~$k$ matrices. We then have $n=d^2$, $n_1=d^2-(d-k)^2$, and $n_2=(d-k)^2$, see \cite[Proposition~4.1]{Helmke/Shayman:1992}. Note that $M$ is not closed in $\R^n$, but only locally closed.

\begin{definition} \label{D:tube}
A {\em tube segment around $\overline x \in M$} is a neighborhood $U$ of $\overline x$ in $\R^n$ together with an open ball $A\subset\R^{n_1}$ and a diffeomorphism
\[
\Phi: A\times B_1 \to U,
\]
where $B_1$ is the open unit ball in $\R^{n_2}$, such that
\begin{enumerate}
\item $\Phi(A\times\{0\}) = M\cap U$,
\item $M\cap U$ has compact closure in $M$,
\item $U$ does not intersect the frontier of $M$---that is, $U\cap (\overline M\setminus M)=\emptyset$,
\item $\Phi$ and $\Phi^{-1}$ have bounded derivatives.
\end{enumerate}
\end{definition}

\begin{proposition} \label{P:tube}
For any $\overline x\in M$, there exists a tube segment around $\overline x$.
\end{proposition}

\begin{proof}
Let $NM$ denote the normal bundle of $M$, and $\chi:NM\to\R^n$ the addition map $\chi(x,v)=x+v$. Consider a {\em tubular neighborhood} of $M$, i.e.~the diffeomorphic image under $\chi$ of a set of the form
\[
T = \{(x,v)\in NM : \|v\| < \rho(x)\},
\]
with $\rho:M\to\R$ strictly positive and continuous. Tubular neighborhoods exist by \cite[Theorem~10.19]{Lee:2003}. Choose a neighborhood $V$ of $\overline x$ in $M$. Shrinking $V$ if necessary, we may assume that $V$ has compact closure in $M$, and that there exists a diffeomorphism $\psi:V\to A$ for some open ball $A\subset\R^{n_1}$. Now, set $\varepsilon=\inf_{x\in V}\rho(x)>0$ and define
\[
T_{V,\varepsilon} = \{ (x,v) \in NM : x\in V,\ \|v\| < \varepsilon \}.
\]
Let $U=\chi(T_{V,\varepsilon})$. Finally, consider the diffeomorphism $\varphi:T_{V,\varepsilon}\to A\times B_1$, $(x,v)\mapsto(\psi(x),\varepsilon^{-1}v)$. To summarize, we have the diffeomorphisms
\[
U \quad\stackrel{\chi^{-1}}{\longrightarrow}\quad
T_{V,\varepsilon} \quad\stackrel{\varphi}{\longrightarrow}\quad
A\times B_1,
\]
and thus define $\Phi = \chi\circ\varphi^{-1}$. It remains to check properties~(i)--(iv). For (i), we compute $\Phi^{-1}(M\cap U)=\varphi\circ\chi^{-1}(M\cap U)=\varphi(\{(x,0):x\in V\})=V\times\{0\}$. Property~(ii) is immediate since $M\cap U=V$ has compact closure in $M$ by construction. For~(iii), note that the inclusion $U\subset \chi(T)$ holds, and that the latter set does not intersect~$\overline M\setminus M$. Finally, by shrinking $A$ and $B_1$, and then applying a homothety to recover $B_1$, allows one to assume that (iv) holds.
\end{proof}




\providecommand{\bysame}{\leavevmode\hbox to3em{\hrulefill}\thinspace}
\providecommand{\MR}{\relax\ifhmode\unskip\space\fi MR }
\providecommand{\MRhref}[2]{%
  \href{http://www.ams.org/mathscinet-getitem?mr=#1}{#2}
}
\providecommand{\href}[2]{#2}

\end{document}